\documentclass[a4paper,12pt, oneside]{scrartcl}

\usepackage{setspace}
\usepackage{tikz}
\usepackage{adjustbox}
\usepackage{tikz-cd}
\usetikzlibrary{arrows}
\usepackage{amssymb,amsmath,amsthm, dsfont}
\usepackage{bm}
\usepackage{mathtools}
\usepackage{upgreek}
\usepackage{graphicx}
\usepackage{amsfonts}
\usepackage{latexsym}
\usepackage[english]{babel}
\usepackage{verbatim} 
\usepackage[dvipsnames]{xcolor}
\usepackage{subcaption}
\usepackage{caption}
\usepackage{amsmath}
\usepackage{amsthm}
\usepackage{mathrsfs}
\usepackage{stmaryrd}
\usepackage{ragged2e}
\usepackage[bookmarks, colorlinks, breaklinks]{hyperref}
\hypersetup{
	linkcolor=black,
	citecolor=teal,
	filecolor=olive,
	urlcolor=olive }
\usepackage[a4paper,top=2.7cm,bottom=2.7cm,left=2.7cm,right=2.7cm]{geometry}
\usepackage{chapterbib}
\usepackage[T1]{fontenc}
\usepackage[utf8]{inputenc}
\usepackage{babel}
\usepackage{blindtext}

\newenvironment{proofc}{\textit{Proof of the claim.}}{}

\newtheorem{lemma}{Lemma}[section]
\newtheorem{theorem}[lemma]{Theorem}
\newtheorem{proposition}[lemma]{Proposition}

\newtheorem*{theorem*}{Theorem}

\theoremstyle{definition}
\newtheorem{definition}[lemma]{Definition}

\newtheorem{notation}[lemma]{Notation}

\theoremstyle{remark}
\newtheorem{example}[lemma]{Example}
\newtheorem{remark}[lemma]{Remark}
\newtheorem*{claim}{Claim}
\newtheorem*{example*}{Example}

\newcommand{\Z}{\mathbb Z}
\newcommand{\Q}{\mathbb Q}

\newcommand{\R}{\mathbb R}

\newcommand{\C}{\mathbb C}
\newcommand{\Zp}{\mathbb Z_p}
\newcommand{\Fp}{\mathbb F_p}

\newcommand{\Qp}{\mathbb Q_p}

\newcommand{\Gm}{\mathbb G{_m}}
\newcommand{\A}{\mathbb A}
\newcommand{\spec}{\text{Spec}}

\newcommand{\AHS}{\mathcal{A}_{1, \text{HS}}^\text{GSp}}
\newcommand{\ASh}{\mathcal{A}_{1, \text{Sh}}^\text{GSp}}
\newcommand{\AShGL}{\mathcal{A}_{1}^\text{GL}}

\newcommand{\AGL}{\mathcal{A}_0^\text{GL}}
\newcommand{\AGSp}{\mathcal{A}_0^\text{GSp}}

\newcommand\blfootnote[1]{%
  \begingroup
  \renewcommand\thefootnote{}\footnote{#1}%
  \addtocounter{footnote}{-1}%
  \endgroup
}

\setkomafont{disposition}{\normalfont\bfseries}

\title{On the geometry of integral models of Shimura varieties with $\Gamma_1(p)$-level structure}
\author{Giulio Marazza}
\date{}

\begin{document}

\maketitle

\begin{abstract}
    \small{\textsc{ABSTRACT}}.\, We study integral models of some Shimura varieties with bad reduction at a prime $p$, namely the Siegel modular variety and Shimura varieties associated with some unitary groups. We focus on the case where the level structure at $p$ is given by the pro-unipotent radical of an Iwahori subgroup, and we analyze the geometry of the integral models that have been proposed until now: we show that they are almost never normal and in some cases not flat over $\Zp$. We do so by showing the failure of these geometric properties on the corresponding local models, and we explain how the local model diagrams can be interpreted using the root stack construction.
\end{abstract}

\hypersetup{
	linkcolor=Mahogany,
	citecolor=teal,
	filecolor=olive,
	urlcolor=olive }

\section*{Introduction}

\blfootnote{giulio.marazza@uni-due.de}
In order to study arithmetic properties of Shimura varieties, it is of interest to have integral models over the ring of integers $\mathcal{O}_E$, where $E$ is the completion of the reflex field $\textbf{E}$ at some finite place $v$ over a prime number $p$. The behavior of the integral model typically changes depending on the choice of level subgroup that one makes. For example, in their book \cite{RZ} Rapoport and Zink define integral models of PEL Shimura varieties with parahoric level structure at $p$: writing $(\textbf{G},\textbf{X})$ for the Shimura datum, this means that the level subgroup splits as $K=K^pK_p$, with $K^p\subset \textbf{G}(\A_f^p)$ small enough and $K_p\subset\textbf{G}(\Qp)$ a parahoric subgroup. In this thesis we choose $K_p$ to be a pro-$p$ Iwahori subgroup, i.e. the pro-unipotent radical of an Iwahori subgroup. We call this a $\Gamma_1(p)$-level structure, whereas in the Iwahori case we speak of $\Gamma_0(p)$-level structure. We consider the Siegel modular varieties, for which integral models have been constructed by Haines and Stroh in \cite{HS} and by Shadrach \cite{Sha}, and some unitary Shimura varieties, for which integral models were constructed by Haines and Rapoport in \cite{HR} and by Shadrach in \cite{Sha}. Let us also mention \cite{Liu}, where Liu generalizes the work of Haines and Stroh to the Hilbert-Siegel case. There is also work of Harris and Taylor in \cite{HT}, in which they consider unitary Shimura varieties, and of Pappas in \cite{Pap}, where he considers Hilbert modular varieties. These models are defined over the corresponding integral model with Iwahori level structure by parameterizing generators of finite locally free group schemes of rank $p$, that arise naturally from the Iwahori moduli problem. We make the assumption that $p\neq 2$, since in this case the $\Gamma_0(p)$ and $\Gamma_1(p)$ level structures coincide. In the Siegel case, the Iwahori integral model $\mathcal{A}_0^\text{GSp}$ parametrizes chains of degree $p$ isogenies between abelian schemes
\[
A_0\longrightarrow A_1\longrightarrow \cdots \longrightarrow A_n,
\]
with some extra structure. The kernels $G_i$ of these isogenies are finite locally free group schemes of rank $p$: by the main result of \cite{OT}, such group schemes are classified by triples $(\mathcal{L},a,b)$ where $\mathcal{L}$ is a line bundle and $a\in\Gamma(\mathcal{L}^{p-1})$, $b\in\Gamma(\mathcal{L}^{1-p})$ are global sections satisfying $a\otimes b =w_p$ for a specific $w_p\in p\Zp^\times$: then, as in \cite[Section 3.3]{HR}, an Oort-Tate generator is a global section $x$ of $\mathcal{L}$ such that $x^{\otimes (p-1)}=a$. In \cite{Sha}, Shadrach defines a pro-$p$ Iwahori integral model $\ASh$ by extracting an Oort-Tate generator for each $G_i$ plus a $(p-1)$-th root of $p$. In \cite{HS}, Haines and Stroh construct another integral model $\AHS$ by parameterizing Oort-Tate generators $g_i,g_i^*$ of both $G_i$ their Cartier duals $G_i^*$, and imposing that $g_ig_i^*$ is independent on $i$; one can find this model also in Liu's work \cite{Liu}. The other case that we consider are unitary Shimura varieties, where the Shimura datum is induced by a division algebra with involution $(D,*)$, whose center $F$ is an imaginary quadratic extension of $\Q$, and an $\R$-algebra homomorphism $h_0:\C\longrightarrow D_\R$. The Iwahori integral model $\mathcal{A}_0^\text{GL}$ still represents the moduli space of chains of isogenies between abelian schemes with extra structure, but this time the degree of these isogenies is a power or $p$. Nevertheless, under suitable assumptions on the behavior at $p$ of the division algebra $D$, one can still find finite flat group schemes of rank $p$, coming from the associated chain of $p$-divisible groups, and then define the pro-$p$ Iwahori integral model $\AShGL$ by extracting Oort-Tate generators of these groups. We call these group schemes $G_i$ both in the unitary and in the Siegel case, and write $(\mathcal{L}_i,a_i,b_i)$ for their Oort-Tate parameters.\\
In order to study étale-local geometric properties of these integral models, it is very fruitful to consider their local model diagrams. In the Iwahori case, this is a smooth correspondence of the form
\[
\begin{tikzcd}
    & \widetilde{\mathcal{A}}_0 \arrow[dl, swap, "\phi"] \arrow[dr, "\psi"] & \\
    \text{M}_0 && \mathcal{A}_0
\end{tikzcd}
\]
where $\psi$ is a $\mathcal{G}$-torsor for a suitable (smooth) group scheme over $\Zp$ which acts also on $\text{M}_0$ and $\phi$ is smooth and $\mathcal{G}$-equivariant. Thus, local questions such as flatness or normality can be investigated on $\text{M}_0$, which is defined in linear algebraic terms and hence more amenable to computations. To construct a local model diagram in the pro-$p$ Iwahori case, one would like to find finite flat group schemes over $\text{M}_0$ having the same pullback to $\widetilde{\mathcal{A}}_0$ as the ones used to construct the $\Gamma_1(p)$ model, so that we could parametrize Oort-Tate generators for them. Unfortunately, on $\text{M}_0$ we can only find line bundles $\omega_i$ corresponding to $\mathcal{L}_i^{p-1}$, by looking at the Fitting ideals of $\omega_{G_i}$: these line bundles and their duals also carry sections which get identified with $a_i$ and $b_i$. However, in order to extract a root of a section of a line bundle, one needs a root of the line bundle itself: this means that on the local model side we have to simultaneously trivialize the $\omega_i$'s. In \cite{Sha} Shadrach does this by restricting to an open cover where the $\omega_i$'s are all trivial, while Liu  in \cite{Liu} considers the $\mathbb{G}_m^n$-torsor of trivializations of the $\omega_i$'s. Both constructions can be interpreted as atlases of (a fiber product of) root stacks of $\text{M}_0$ relative to the line bundles $\omega_i$, as we show in Chapter \ref{Chap local models}. We then use these local models to prove the following result about the geometry of the corresponding integral models in Chapter \ref{chap norm flat}.

\begin{theorem*}[\ref{non normal},\;\ref{Teo HS}]
    \begin{itemize}
        \item [$(1)$] The integral model $\mathcal{A}_1^\emph{GL}$ is not normal unless the signature of the unitary group is $(n-1,1)$, in which case it is regular.
        \item [$(2)$] The integral model $\mathcal{A}_{1,\emph{HS}}^\emph{GSp}$ is topologically flat but not flat over $\Zp$, except for the elliptic curve case. 
    \end{itemize}
\end{theorem*}

It is perhaps surprising that the unitary $\Gamma_1(p)$-models with signature $(r,n-r)$ and $(n-r,r)$ are not isomorphic, since this happens at Iwahori level. This is due to the fact that Oort-Tate generators do not behave well with respect to Cartier duality: for example, over $\overline{\mathbb{F}}_p$ the constant group $\Z/p\Z$ has $p-1$ Oort-Tate generators (all the non-zero sections), while the only Oort-Tate generator of $\mu_p$ is the zero section. Another difference with the $\Gamma_0(p)$-case is that the geometry of these models depends more directly on the prime $p$: for instance, the local model in the unitary case has $g$ irreducible components, where $g=\text{gcd}(p-1,n-r)$, all non-canonically isomorphic to each other. Still in the case of unitary Shimura varieties, we can compute the normalization of an irreducible component of the $\Gamma_1(p)$-local model in the case of signature $(1,n-1)$.

\begin{theorem*}[\ref{normalization C1}]
    The normalization of an irreducible component of the pro-$p$ Iwahori local model in the case of signature $(1,n-1)$ is isomorphic to the spectrum of
    \[
    \frac{\Zp[u_{\underline{c}} \,|\, \underline{c}\in A]}{\begin{pmatrix}
        u_{\underline{c}_1}u_{\underline{c}_2}-u_{\underline{c}_3}u_{\underline{c}_4} \,|\, \scalebox{0.9}{$\underline{c}_1+\underline{c}_2=\underline{c}_3+\underline{c}_4$}\\
        (u_{\underline{c}^{(0)}} \cdots u_{\underline{c}^{(n-1)}})^{h}-p
    \end{pmatrix}},
 \]
where $A=\{\underline{c}\in \Z^n_{\geq 0} \,|\, \sum_ic_i=g \}$. 
\end{theorem*}

\bigskip
\bigskip

I would like to heartly thank my advisor Ulrich Görtz for his encouragement and interest in my work. I would also like to thank Vytautas Paškūnas for a discussion, Thomas Haines for providing a proof of his and Benoıt Stroh, Michael Rapoport for sharing part of an email of Georgios Pappas about the flatness of some integral models and also Jürgen Herzog and Oleksandra Gasanova for some discussions about Veronese subrings. I would like to thank Timo Richarz as well, for accepting to review my thesis, of which this manuscript is part, and which was carried out in the DFG Research Training Group 2553.

\section{The integral models}\label{Chap integral models}

We start by recalling from the literature some standard facts about integral and local models of Shimura varieties, in order to set up the notation used later in the text.

\begin{notation}  
We will use the symbol $\mathcal{A}_i$ for the integral model with $\Gamma_i(p)$-level structure, for $i=0,1$. We will add a superscript $\text{GL}$, respectively $\text{GSp}$, to indicate that we are considering the unitary, respectively Siegel, case.    
\end{notation}

\subsection{The unitary case}

\subsubsection{The Iwahori level}\label{subsection Iwahori level}

Let $F$ be an imaginary quadratic extension of $\Q$, $(D,*)$ a division algebra with center $F$, of dimension $n^2$ over $F$, with an involution $*$ inducing the non-trivial element of $\text{Gal}(F/\Q)$. Let also $h_0:\C\rightarrow D_\R$ be an $\R$-algebra morphism such that $h_0(z)^*=h_0(\bar{z})$ and such that the involution $x\mapsto h_0(i)^{-1}x^*h_0(i)$ is positive definite. Consider the reductive group $\textbf{G}$ over $\Q$ given by
\[
\textbf{G}(R)=\{ x\in(D\otimes_\Q R)^\times \,|\, x^*x\in R^\times \}
\]
and let $\textbf{X}$ be $\textbf{G}(\R)$-conjugacy class of the inverse of $h_0$ restricted to $\C^\times$: this gives the unitary Shimura datum $(\textbf{G},\textbf{X})$. Assume that $p$ splits in $F$ as $(p)=\mathfrak{p}\mathfrak{p}^*$ and that $D$ splits over $\Qp$:
\[
D_{\Qp}=D_\mathfrak{p}\times D_{\mathfrak{p}^*}=M_n(\Qp)\times M_n(\Qp),
\]
so that the group $\textbf{G}_{\Qp}$ is isomorphic to $\text{GL}_n\times \Gm$. In order to define the integral model for the unitary Shimura varieties with Iwahori level structure at $p$, we need to specify some extra integral data: namely a $\Z_{(p)}$-order $\mathcal{O}_B\subset D^{op}$ and a self dual multichain of $\mathcal{O}_B$-lattices. We choose $\mathcal{O}_B$ to be the unique maximal $\Z_{(p)}$-order such that $\mathcal{O}_B\otimes \Zp$ corresponds to $\text{M}_n(\Zp)\times \text{M}_n(\Zp)$ under the isomorphism $D_{\Qp}\cong \text{M}_n(\Qp)\times \text{M}_n(\Qp)$; for the choice of the multichain we refer to \cite[Section 5.2]{Hai} .

\begin{definition}[\cite{RZ}, Definition 6.9]\label{def unitary Iwahori}
    The integral model of the Shimura variety relative to the above Shimura datum with Iwahori level structure at $p$ is the scheme $\mathcal{A}_0^\text{GL}$ representing the functor which sends a $\Zp$-scheme $S$ to the set of isomorphism classes of tuples $(A_\bullet, \overline{\lambda}, \overline{\eta})$, where:
    \begin{itemize}
        \item [$(i)$] $A_\bullet$ is a chain $A_0\xrightarrow{\alpha_0}A_1\xrightarrow{\alpha_1}\cdots\xrightarrow{\alpha _{n-1}}A_0$ of $n^2$-dimensional abelian schemes over $S$, with $\alpha_i$ isogenies of degree $p^{2n}$ such that the composition $\alpha_{n-1}\cdots\alpha_0$ is multiplication by $p$;
        \item [$(ii)$] each $A_i$ is equipped with an $\mathcal{O}_B\otimes \Z_{(p)}$-action commuting with the isogenies $\alpha_i$;
        \item [$(iii)$] the action of $\mathcal{O}_B\otimes \Z_{(p)}$ satisfies the Kottwitz condition;
        \item [$(iv)$] $\overline{\lambda}$ is a $\Q$-homogeneous class of principal polarizations;
        \item [$(v)$] $\overline{\eta}$ is a $K^p$-level structure.
    \end{itemize}
\end{definition}

To this integral model one can associate a local model diagram, i.e. a smooth correspondence of the form
\begin{equation}\label{locmod diagram GL}
\begin{tikzcd}
& \tilde{\mathcal{A}_0}^\text{GL} \arrow[ld, swap, "\phi"] \arrow[rd, "\psi"] & \\
\text{M}^\text{GL}_0 & & {\mathcal{A}_0^\text{GL}},
\end{tikzcd}
\end{equation}
where $\tilde{\mathcal{A}}_0$ represents the functor sending a $\Zp$-algebra $R$ to the set of pairs
\[
(A_\bullet, \lambda, \overline{\eta})\in \mathcal{A}_0(R); \;\;\; \gamma:H^1_{dR}(A_\bullet)\rightarrow\Lambda_\bullet\otimes_{\Zp}R.
\]
Here $\Lambda_\bullet=(\Lambda_0\subset\Lambda_1\subset\cdots\subset\Lambda_{n-1}\subset p^{-1}\Lambda_0)$ is the lattice chain inside the product $M_n(\Qp)\times M_n(\Qp)$ mentioned above and $\gamma$ is an isomorphism of lattice chains. The scheme $\text{M}^\text{GL}_0$ is called the {local model} for $\AGL$, and it represents the functor sending a $\Zp$-algebra $R$ to the set of commutative diagrams of the form
\[
\begin{tikzcd}
    R^{n} \arrow[r, "\varphi_0"] &R^{n} \arrow[r, "\varphi_1"] &\cdots \arrow[r, "\varphi_{n-2}"] &R^{n} \arrow[r, "\varphi_{n-1}"] &R^{n}\\
    F_0 \arrow[r] \arrow[u, hook] &F_1 \arrow[r] \arrow[u, hook] &\cdots \arrow[r] &F_{n-1} \arrow[r] \arrow[u, hook] &F_0, \arrow[u, hook]
\end{tikzcd}
\]
where $\varphi_i=\text{diag}(1,\dots, p,\dots, 1)$ with $p$ in the $(i+1)$-th component and $F_i$ are direct summands of $R^{n}$ of rank $r$. The general definition of local model can be found in \cite{RZ}, while in \cite{Goer} one can find a more explicit description. The map $\psi$ appearing in the diagram is defined as
\begin{align*}
    \psi: ((A_\bullet, \lambda, \overline{\eta}), \gamma)&\mapsto (A_\bullet, \lambda, \overline{\eta}),
\end{align*}
while $\phi$ is defined by taking the image of $\omega_{A_\bullet}$ under $\gamma$ and then using Morita equivalence (see \cite[Section 6.3.3]{Hai}). Here, $\omega_{A_\bullet}$ is the chain of Hodge bundles coming from the chain of abelian schemes $A_\bullet$.

\paragraph{The Iwahori local model for GL}

Given $F_\bullet \in \text{M}_0^\text{GL}(R)$, we use the notations $\bar{\varphi}_{i}$ and $\varphi_{i _{|F_i}}$ for the following maps:

\[
\begin{tikzcd}
    \overline{F}_0 \arrow[r, "\bar{\varphi}_0"] & \overline{F}_1 \arrow[r] &\cdots \arrow[r] &\overline{F}_{n-1} \arrow[r, "\bar{\varphi}_{n-1}"] &\overline{F}_0 \\
    R^n \arrow[r, "\varphi_0"] \arrow[u, ->>] &R^n \arrow[r] \arrow[u, ->>] &\cdots \arrow[r] &R^n \arrow[r, "\varphi_{n-1}"] \arrow[u, ->>] &R^n \arrow[u, ->>]\\
    F_0 \arrow[r, "\varphi_{0 _{|F_0}}"] \arrow[u, hook] &F_1 \arrow[r] \arrow[u, hook] &\cdots \arrow[r] &F_{n-1} \arrow[r, "\varphi_{n-1 _{|F_{n-1}}}"] \arrow[u, hook] &F_0 \arrow[u, hook].
\end{tikzcd}
\]

The scheme $\text{M}^\text{GL}_0$ is a closed subscheme of a product of Grassmannians over $\Zp$, hence it is projective, and its generic fiber is isomorphic to $\text{Grass}_r(\Zp^n)$. Its special fiber can be embedded into $\mathscr{F}_{\text{GL}_n}$, the affine flag variety for $\text{GL}_{n,\Fp}$, as in \cite[Section 4.2]{Goer}: this relies on the fact that $\mathscr{F}_{\text{GL}_n}$ can be described as a moduli space of complete lattice chains. It can also be defined as the fppf-quotient $L\text{GL}_n/L^+I$, where $L\text{GL}_n$ is the loop group of $\text{G}_{n,\Fp}$ and $L^+I$ is an Iwahori subgroup, consisting of matrices which are upper triangular modulo $t$. The ind-group scheme $L^+I$ acts by multiplication on the left on the affine flag variety: the orbits of this action are called Schubert cells. They are locally closed subschemes and can be indexed by the extended affine Weyl group of $\text{GL}_n$, which is the semidirect product $\widetilde{W}=X_*(T)\rtimes W$ where $X_*(T)$ is the cocharacter lattice of a maximal torus $T$ of $\text{GL}_n$ and $W$ is its Weyl group: we write elements of $\widetilde{W}$ as $t_\nu\bar{w}$ for $\nu\in X_*(T)$ and $\bar{w}\in W$. We can embed $\widetilde{W}$ inside $\text{GL}_n(\Fp(\!(t)\!))$ as follows: we identify $\lambda\in X_*(T)$ with $\lambda(t)\in \text{GL}_n(\Fp(\!(t)\!))$ and $W$ with the group of permutation matrices inside $\text{GL}_n(\Fp)\subset \text{GL}_n(\Fp(\!(t)\!))$. The Schubert cell associated with $w\in\widetilde{W}$ is the $L^+I$-orbit of $w$ seen inside $\mathscr{F}_{\text{GL}_n}(\Fp)$. There is another convenient way of parameterizing the Schubert cells using alcoves for $\text{GL}_n$ that we now explain.

\begin{definition}[\cite{KR}, Section 3.2]
    An alcove $x=(x_i)_i$ for $\text{GL}_n$ is a collection of $n$ vectors $x_0, \cdots, x_{n-1}\in \Z^n$ such that:
    \begin{itemize}
        \item[$(i)$] $x_0\leq x_1\leq \dots \leq x_{n-1}\leq x_n=x_0 + \textbf{1}$, where inequality is component-wise and $\textbf{1}=(1, \ldots, 1)\in \Z^n$.
        \item[$(ii)$] $\sum_j x_{i+1}(j) = \sum_j x_{i}(j) +1$
    \end{itemize}
        We say that an alcove $x$ has size $r$ if $\sum_j x_0(j)=r$. For an alcove $x=(x_i)_i$, we define its difference vectors to be $t_i^x=x_i-\omega_i$, where $\omega_i=(1^{(i)}, 0^{(n-i)})$ 
\end{definition}

We now choose $T$ to be the standard diagonal maximal torus of $\text{GL}_n$, so that $\widetilde{W}$ gets identified with $\Z^n\rtimes S_n$, with $S_n$ permuting the coordinates of $\Z^n$. The extended Weyl group acts simply transitively on the set of alcoves via
\[
(\lambda \sigma)\cdot (x_i)_i= (\lambda +\sigma\cdot x_i)_i
\]
so choosing as base point the alcove $\omega=(\omega_i)_{i=0}^{n-1}$ we can identify $\widetilde{W}$ with the set of alcoves: we write $w_x$ for the element of $\widetilde{W}$ corresponding to the alcove $x$. Any alcove $x$ determines an $\Fp$-valued point $\mathscr{L}_\bullet^{x}\in \mathscr{F}_{\text{GL}_n}$, given by
\[
\mathscr{L}^x_i=\langle t^{-x_i(1)+1}\text{e}_1, \dots, t^{-x_i(n)+1}\text{e}_n \rangle_{\Fp\llbracket t\rrbracket}\subset \Fp(\!(t)\!)^n.
\]
For example, the alcove $\omega=(\omega_i)_i$ corresponds to $t\lambda_\bullet$ and one sees that
\[
t\lambda_\bullet \subset \mathscr{L}^x_\bullet \subset \lambda_\bullet \hspace{0.3cm} \Leftrightarrow \hspace{0.3cm} \omega_i \leq x_i \leq \omega_i +\textbf{1} \hspace{0.2cm}\forall i=0,\ldots, n-1:
\]
alcoves satisfying this condition are called minuscule.

\begin{definition}[\cite{KR}]
    Let $x$ be an alcove for $\text{GL}_n$. Then $x$ is said to be:
    \begin{itemize}
        \item [$(i)$] $r$-permissible if it is minuscule of size $r$;
        \item [$(ii)$] $r$-admissible if there exists $\sigma\in W$ such that $w_x\leq \sigma(\omega_r)$ in the Bruhat order of $\widetilde{W}$.
    \end{itemize}
\end{definition}

The main result in the paper of Kottwitz and Rapoport is that an alcove is $r$-permissible if and only if it is $r$-admissible, see \cite[Theorem 3.5]{KR}. One can define for any cocharacter $\lambda\in X_*(T)$ its admissible subset as
\[
\text{Adm}(\lambda)=\{ w\in \widetilde{W} \;|\; \exists\, \tau \in W \; \text{such that} \; w\leq t_{\tau(\lambda)} \}:
\]
then $x$ is an $r$-admissible alcove if and only if $w_x\in\text{Adm}(\mu)$ for $\mu=(1^{(r)},0^{(n-r)})$. From this discussion, it is clear that we can stratify the special fiber of the local model as
\[
\text{M}_{0,\Fp}^\text{GL}=\bigcup_{w\in \text{Adm}(\mu)} S_w
\]
where $S_w$ is the stratum corresponding to $w$: this is the so-called Kottwitz-Rapoport stratification. Since its strata are $\mathcal{G}$-orbits, we can descend it through the local model diagram (\ref{locmod diagram GL}) to a stratification on $\mathcal{A}^\text{GL}_{0,\Fp}$: we will denote by $\mathcal{A}_{0,w}$ the stratum corresponding to $w$. The embedding of the special fiber of the local model into the affine flag variety is crucial in the proof of the flatness of $\text{M}_0^\text{GL}$ by Görtz in \cite{Goer}. Let us also mention that $\text{M}_0^\text{GL}$ is normal, thanks to \cite{He}. We finish this subsection by noting that to each $r$-permissible alcove $x$ we can associate an open subscheme of $\text{M}_0^\text{GL}$ defined by

\begin{equation}\label{open unitary}
    U_x(R)=\{ F_\bullet \in \text{M}^\text{GL}_0(R) \;|\; F_i\oplus ( \bigoplus_{t_i^x(j)=0} R\text{e}_j)=R^n \},
\end{equation}

where $t^x_i=x_i-\omega_i$, the difference vectors of $x$. These opens cover the local model and, by \cite[Proposition 4.5]{Goer}, $U_x$ contains the stratum $S_{w_x}$.

\subsubsection{The pro-\texorpdfstring{$p$}{} Iwahori level}

Before giving the definition of the integral model with $\Gamma_1(p)$-level structure, we need to introduce the notion of Oort-Tate generators, coming from the classification of finite locally free group schemes of rank $p$ established in \cite{OT}. First, we present the summary of Oort-Tate theory as in \cite[Theorem 3.3.1]{HR}.

\begin{theorem}[Tate-Oort]
    Let $OT$ be the $\Zp$-stack of finite locally free group schemes of rank $p$.
    \begin{itemize}
        \item [$(i)$] OT is an Artin stack isomorphic to
        \[
            OT=\left[ \emph{Spec}\left(\frac{\Zp[x,y]}{(xy-\omega_p)}\right)/\Gm \right]
        \]
        where $\Gm$ acts by $t\cdot (x,y)=(t^{p-1}x, t^{1-p}y)$ and $\omega_p$ is an explicit element of $p\Z_p^\times$.
        
        \item [$(ii)$] The universal group scheme $\mathcal{G}_{OT}$ over OT is given by
        \[
            \mathcal{G}_{OT}=\left[ \emph{Spec}\left(\frac{\Zp[x,y,z]}{(xy-\omega_p, z^p-xz)}\right)/\Gm \right]
        \]
        where $\Gm$ acts as before on $x, y$ and with weight $1$ on $z$, with zero section $z=0$.

        \item[$(iii)$] Cartier duality acts on $OT$ by interchanging $x$ and $y$.
    \end{itemize}
\end{theorem}

\begin{remark}\label{OT decomposition}
In \cite[Lemma 2]{OT}, Tate and Oort prove that the augmentation ideal of a finite locally free group scheme $G$ of rank $p$ over a $\Lambda$-scheme $S$ decomposes as $I=\bigoplus_{i=1}^{p-1}I_1^{i}$, where $I_1\subset I$ is a line bundle over $S$. They consider the map $a:I_1^{\otimes p}\rightarrow I_1$, induced by the multiplication in the Hopf algebra of $G$, and the map $a':{I'}_1^{\otimes p}\rightarrow I'_1$, where $I'$ is the augmentation ideal of the Cartier dual of $G$. The equivalence is then given by sending $G$ to the triple $(I_1',a,a')$: we call this triple the Oort-Tate parameters of $G$. Note that if $(\mathcal{L},a,b)$ are the Oort-Tate parameters for $G$, then $(\mathcal{L}^\vee,b^\vee,a^\vee)$ are the parameters for the Cartier dual $G^*$.   
\end{remark}

Now consider the closed substack
$$\mathcal{G}_{OT}^\times = \left[ \text{Spec}\left(\frac{\Zp[x,y,z]}{(xy-\omega_p, z^{p-1}-x)}\right)/\Gm \right]$$
of $\mathcal{G}_{OT}$, called the substack of generators of $\mathcal{G}_{OT}$.

\begin{definition}[\cite{HR}, Section 3.3]
    Let $S$ be a $\Z_p$-scheme, $G$ a finite locally free group scheme of rank $p$ over $S$. Let $G^\times$ be the closed subscheme of $G$ obtained by pulling back $\mathcal{G}_{OT}^\times\subset \mathcal{G}_{OT}$ along the map $S\xrightarrow{}OT$ corresponding to $G$. A global section $z\in G(S)$ is called an {Oort-Tate generator} if it factors through $G^\times$:
    \[
    \begin{tikzcd}
    G^\times \arrow[d] \arrow[r] & G \arrow[d] \arrow[r] & S \arrow[d] \arrow[ll, "z"', bend right] \\
    \mathcal{G}_{OT}^\times \arrow[r] & \mathcal{G}_{OT} \arrow[r] & OT  
    \end{tikzcd}
    \]
\end{definition}

\begin{remark}
    As shown in \cite[Remark 3.3.2]{HR}, $z\in G(S)$ is an Oort-Tate generator if and only if its multiples $\{0,c,[2]c,\ldots,[p-1]c\}$ form a ``full set of sections'' in the sense of Katz and Mazur, \cite[1.8.2]{KM}. When the base $S$ lives over $\Qp$ then $G$ is étale and, by \cite[lemma 1.8.5]{KM}, on geometric points the notion of Oort-Tate generator coincides with the usual notion of point of order $p$ of an abstract group. This will ensure that the $\Gamma_1(p)$-integral models have the correct generic fiber.
\end{remark}

In order to construct the integral model in the unitary case, we need to associate to each point $(A_\bullet, \overline{\lambda}, \overline{\eta})$ of $\mathcal{A}_0$ a tuple of finite locally free group schemes $G_i$ of rank $p$, as in \cite[Section 3.1]{Sha}. To do so, note that the $p$-divisible group $A_i[p^\infty]$ of each member of the chain $A_\bullet$ carries an action of $\mathcal{O}_B\otimes\Zp\cong M_n^\text{op}(\Zp)\times M_n^\text{op}(\Zp)$ and denote by $\text{e}_{11}\in M_n^\text{op}(\Zp)\times M_n^\text{op}(\Zp)$ the idempotent on the first factor. Applying it to the chain of isogenies between the $p$-divisible groups of the abelian varieties $A_i$, we obtain a chain
\[
\text{e}_{11}A_0[p^\infty]\rightarrow \text{e}_{11}A_1[p^\infty]\rightarrow\cdots \rightarrow \text{e}_{11}A_{n-1}[p^\infty]\rightarrow \text{e}_{11}A_0[p^\infty]
\]
of degree $p$ isogenies between $p$-divisible groups of height $n$, whose composition is multiplication by $p$. We define the finite locally free group scheme of rank $p$
\begin{equation}\label{Gi}
G_i=\text{ker}(\text{e}_{11}A_i[p^\infty]\rightarrow \text{e}_{11}A_{i+1}[p^\infty]):
\end{equation}
the collection of these groups for each $i=0,\ldots, n-1$ gives a map $\mathcal{A}_0^\text{GL}\xrightarrow{\varphi} OT^n$.

\begin{definition}[\cite{Sha}, Definition 3.8]
    The integral model $\AShGL$ is defined as the following pullback:
    \[
    \begin{tikzcd}
            {\mathcal{A}_{1}^\text{GL}} \arrow[d] \arrow[r] & (\mathcal{G}_{OT}^\times)^n \arrow[d] \\
            {\mathcal{A}_0^\text{GL}} \arrow[r, "\varphi"] & (OT)^n.
    \end{tikzcd}
    \]
\end{definition}

Thus, points of $\AGL$ over $(A_\bullet, \bar{\lambda}, \bar{\eta})\in \AShGL$ parametrize Oort-Tate generators for the groups $G_i$.

\subsection{The Siegel case}

In the Siegel case there are two constructions of $\Gamma_1(p)$-integral models, one by Shadrach in $\cite{Sha}$ and one by Haines and Stroh in their unpublished work \cite{HS}: we present both of them and discuss their relation.

\subsubsection{The Iwahori level}\label{subsection Iwahori GSp}

\begin{definition}[\cite{RZ}, Definition 6.9]\label{GSp iwahori integral}
    The integral model for the Siegel modular variety with Iwahori level structure at $p$ is the scheme $\mathcal{A}_0^\text{GSp}$ representing the functor which sends a $\Z_p$-scheme $S$ to the set of isomorphism classes of tuples $(A_\bullet, \lambda_0, \lambda_n, \overline{\eta})$, where:
    \begin{itemize}
        \item [$(i)$]$A_\bullet$ is a chain $A_0\xrightarrow{\alpha_0}A_1\xrightarrow{\alpha_1}\cdots\xrightarrow{\alpha _{n-1}}A_n$ of $n$-dimensional abelian schemes over $S$, with $\alpha_i$ isogenies of degree $p$;
        \item [$(ii)$] the maps $\lambda_0, \lambda_n$ are principal polarizations of $A_0, A_n$ respectively, making the loop starting at any $A_i$ or $A_i^\vee$ in the diagram
        \[
        \begin{tikzcd}
            A_0 \arrow[r, "\alpha_0"] & A_1 \arrow[r,"\alpha_1"] & \cdots \arrow[r,"\alpha_{n-1}"] & A_{n} \arrow[d,"\lambda_n"] \\
            A_0^\vee \arrow[u, "\lambda_0^{-1}"] & A_1^\vee \arrow[l, "\alpha_0^\vee"] & \cdots \arrow[l, "\alpha_1^\vee"] & A_{n}^\vee \arrow[l, "\alpha_{n-1}^\vee"]
        \end{tikzcd}
        \]
        multiplication by $p$;
        \item [$(iii)$] $\overline{\eta}$ is a $K^p$-level structure.
    \end{itemize}
\end{definition}

There is a local model diagram associated to this integral model as well: 
\[
    \begin{tikzcd}
        & \tilde{\mathcal{A}_0}^\text{GSp} \arrow[ld, swap, "\phi"] \arrow[rd, "\psi"] & \\
        \text{M}^\text{GSp}_0 & & {\mathcal{A}_0}^\text{GSp}.
    \end{tikzcd}
\]
Before explaining the terms appearing in this diagram, let us write $(\cdot,\cdot)$ for the standard symplectic pairing on $\Zp^{2n}$ given by the matrix   $$J_{2n}=\begin{pmatrix}
    & \textbf{J} \\
    -\textbf{J} &
\end{pmatrix},$$ where $\textbf{J}=(\delta_{n-i+1,j})_{ij}$. For $i=0,\ldots, 2n-1$, let $\Lambda_i=\langle p^{-1}\text{e}_1,\ldots,, p^{-1}\text{e}_i, \text{e}_{i+1},\ldots, \text{e}_{2n} \rangle_{\Zp}$: this gives rise to a complete lattice chain $\Lambda_\bullet$ in $\Qp^n$
\[
\cdots \rightarrow \Lambda_0\rightarrow \Lambda_1\rightarrow \cdots \Lambda_{2n-1}\rightarrow p^{-1}\Lambda_0\rightarrow \cdots
\]
which is self dual with respect to the pairing $(\cdot,\cdot)$ (in fact, $\Lambda_i^*=p\Lambda_{2n-1}$). For integers $a\leq b$ we write $\Lambda_\bullet^{[a,b]}$ for the lattice chain $(\Lambda_a\subset \cdots \subset \Lambda_b)$. Then, $\tilde{\mathcal{A}}_0$ represents the functor sending a $\Zp$-algebra $R$ to the set of pairs
\[
(A_\bullet, \lambda, \overline{\eta})\in \mathcal{A}_0(R); \;\;\; \gamma:H^1_{dR}(A^+_\bullet)\rightarrow\Lambda_\bullet^{[0,2n]}\otimes_{\Zp}R,
\]
where $\gamma$ is an isomorphism of polarized multichains of $R$-modules and by $A_\bullet^+$ we mean the chain of degree $p$ isogenies
\[
A_0\rightarrow A_1 \rightarrow \overset{\alpha_i}{\cdots} \rightarrow A_n\cong A_n^\vee \rightarrow A_{n-1}^\vee \rightarrow \overset{\alpha_i^\vee}{\cdots} \rightarrow A_0^\vee\cong A_0.
\]
The local model $\text{M}^\text{GSp}_0$ represents the functor sending a $\Zp$-algebra $R$ to the set of commutative diagrams of the form

\begin{equation}
\begin{tikzcd}
    \Lambda_{0,R} \arrow[r] &\Lambda_{1,R} \arrow[r] &\cdots \arrow[r] &\Lambda_{2n-1,R} \arrow[r] &p^{-1}\Lambda_{0,R}\\
    F_0 \arrow[r] \arrow[u, hook] &F_1 \arrow[r] \arrow[u, hook] &\cdots \arrow[r] &F_{2n-1} \arrow[r] \arrow[u, hook] &p^{-1}F_0, \arrow[u, hook]
\end{tikzcd}
\end{equation}

where $F_i$ are direct summands of $\Lambda_{i,R}$ of rank $n$ such that for each $i=0,\ldots,2n-1$ the map 
\begin{equation}\label{condition GSp loc model}
F_i\hookrightarrow \Lambda_{i,R} \cong \Lambda_{2n-i,R}^\vee \rightarrow F_{2n-i}^\vee
\end{equation}
is zero. The two maps appearing in the diagram are defined as:
\begin{align*}
    \psi: ((A_\bullet, \lambda, \overline{\eta}), \gamma)&\mapsto (A_\bullet, \lambda, \overline{\eta})\\
    \phi: ((A_\bullet, \lambda, \overline{\eta}), \gamma)&\mapsto \gamma(\omega_{A^+_\bullet})\subset \Lambda_\bullet\otimes_{\Zp}R.
\end{align*}
Here $\omega_{A_\bullet^+}$ is the associated chain of linear maps between the Hodge bundles of the abelian varieties appearing in $A_\bullet^+$. Since this is a contravariant construction, the maps between Hodge bundles go in the opposite direction, thus $F_i\rightarrow F_{i+1}$ corresponds to $\omega_{A_{i+1}}\rightarrow \omega_{A_i}$.

\paragraph{The local model for GSp}

It is clear that $\text{M}_0^{\text{GSp}_{2n}}$ is a closed subscheme of $\text{M}_0^{\text{GL}_{2n}}$, hence we can restrict the stratification of $\text{M}_{0,\Fp}^{\text{GL}_{2n}}$ to one on $\text{M}_{0,\Fp}^{\text{GSp}_{2n}}$. On the other hand, we can embed the special fiber of $\text{M}_0^{\text{GSp}_{2n}}$ inside an affine flag variety for $\text{GSp}_{2n}$, similarly as in the unitary case (see \cite[Chapter 5]{Goer2}), obtaining a commutative diagram 
\[
\begin{tikzcd}
\text{M}_0^{\text{GL}_{2n}} \arrow[r, hook]                  & \mathscr{F}_{\text{GL}_{2n}}                  \\
\text{M}_0^{\text{GSp}_{2n}} \arrow[r, hook] \arrow[u, hook] & \mathscr{F}_{\text{GSp}_{2n}}. \arrow[u, hook]
\end{tikzcd}
\]
As in the unitary case, $\text{M}_{0,\Fp}^\text{GSp}$ is the union of certain Schubert varieties of $\mathscr{F}_{\text{GSp}_{2n}}$, that can be indexed by permissible alcoves for $G=\text{GSp}_{2n}$ (which we call $G$-alcoves, following the notation of \cite[Section 4]{KR}). We recall their definition:

\begin{definition}[\cite{KR}]
    A $G$-alcove $x=(x_i)_i$ is an alcove for $\text{GL}_{2n}$ satisfying the following condition for some $d\in \Z$:
    \[
    x_{2n-1-i}=\textbf{d} + \theta(x_i)
    \]
    for $i=0,\ldots, 2n-1$. Here $\textbf{d}=(d,\ldots,d)\in \Z^{2n}$ and $\theta$ is the automorphism of $\mathbb{R}^{2n}$ defined by
    \[
    \theta(x_1,x_2,\ldots, x_{2n})=(-x_{2n},\ldots, -x_2, -x_1).
    \]
\end{definition}
The size of a $G$-alcove is defined by regarding it as an alcove for $\text{GL}_{2n}$ and similarly for the notion of minuscule $G$-alcove. The same reasoning as for the $\text{GL}_n$-case shows that the stratification of $\text{M}_{0,\Fp}^\text{GSp}$ is indexed by minuscule $G$-alcoves of size $n$. Finally, to any minuscule $G$-alcove $x$ we can associate an open subscheme $U_x$ of $\text{M}_0^\text{GSp}$ as in (\ref{open unitary}).

\subsubsection{The pro-\texorpdfstring{$p$}{} Iwahori level}

There are two different integral models for the Siegel modular variety with $\Gamma_1(p)$-level structure, proposed respectively by Shadrach in \cite{Sha} and Haines and Stroh in \cite{HS}. We recall their definitions.

\begin{definition}[\cite{Sha},\cite{HS}]
\begin{itemize}
    \item [(1)]Shadrach's integral model $\mathcal{A}_{1, \text{Sh}}^\text{GSp}$ is defined by the following pullback square
        \[
        \begin{tikzcd}
            {\mathcal{A}_{1, \text{S}}^\text{GSp}} \arrow[d] \arrow[r] & (\mathcal{G}_{OT}^\times)^n \arrow[d] \\
            {\mathcal{A}_0^\text{GSp}[t]} \arrow[r, "\varphi"] & (OT)^n,                               
        \end{tikzcd}
        \]
        where $\mathcal{A}_0^\text{GSp}[t]=\underline{\text{Spec}}_{\mathcal{A}_0^\text{GSp}}\left(\frac{\mathcal{O}[t]}{(t^{p-1}-p)}\right)$ and $\varphi$ is the projection to $\mathcal{A}_0^\text{GSp}$ followed by the map corresponding to the kernels of the isogenies $\alpha_0,\ldots,\alpha_{n-1}$. 
    \item[(2)]Haines and Stroh's integral model $\AHS$ is defined by the following pullback square
     \[
        \begin{tikzcd}
            \AHS \arrow[d] \arrow[r] & HS= \left[ \text{Spec}\left(\scalebox{1.3}{$\frac{\Z_p[x_i, y_i, u_i, v_i]_{i=0}^{n-1}}{\left(\substack{\mathstrut x_iy_i-\omega_p, \;u_0v_0-u_iv_i \\ u_i^{p-1}-x_i,\; v_i^{p-1}-y_i}\right)_i}$}\right) \Big/ \mathbb{G}_m^n \right] \arrow[d] \\
            {\mathcal{A}_0^\text{GSp}} \arrow[r, "\varphi"] & (OT)^n\cong \left[ \text{Spec}\left(\scalebox{1.3}{$\frac{\Z_p[x_i, y_i]_{i=0}^{n-1}}{(x_iy_i-\omega_p)_i}$}\right) / \mathbb{G}_m^n \right],         
        \end{tikzcd}
    \]
    where $\mathbb{G}_m^n$ acts by $(t_i)_i\cdot (x_i, y_i, u_i, v_i)_i=(t_i^{p-1}x_i, t_i^{1-p}y_i, t_iu_i, t_i^{-1}v_i)_i$.   
\end{itemize}
\end{definition}

\begin{notation}
We will write $G_i$ both for $\text{Ker}(\alpha_i)$ and the group scheme constructed in (\ref{Gi}), its Cartier dual will be denoted by $G_i^*$. For such $G_i$, we will write $(\mathcal{L}_i, a_i, b_i)$ for its Oort-Tate parameters. Hence $(\mathcal{L}_i^\vee, b_i^\vee, a_i^\vee)$ will be the Oort-Tate parameters for $G_i^*$.

\end{notation}

\begin{remark}\label{points pro-p Iw Siegel}
    Points of these two integral models over $(A_\bullet, \lambda, \overline{\eta})\in \AGSp$ admit the following description:
    \begin{itemize}
        \item[$(a)$] in Shadrach's case, they consist of tuples $(t, g_i)_{i=0}^{n-1}$, where $t$ is a $(p-1)$-th root of $p$ over the base and $g_i$ is an Oort-Tate generator of $G_i$;
        \item[$(b)$] in Haines-Stroh's case, they consist of tuples $(g_i,g_i^*)_{i=0}^{n-1}$, where $g_i$ (respectively $g_i^*$) are Oort-Tate generators of $G_i$ (respectively $G_i^*$) such that $g_ig_i^*$ is independent of $i$.
    \end{itemize}
    These two constructions give the same result on the generic fiber. In fact, there is a finite map
    \begin{align*}
    &\AHS\xrightarrow{\makebox[2cm]{$\rho$}}\ASh \\
    ((A_\bullet, \lambda&, \overline{\eta}), g_i, g_i^*)\mapsto ((A_\bullet, \lambda, \overline{\eta}), g_i, g_0g_0^*) 
    \end{align*}
    which admits an inverse over $\Qp$, namely
    \begin{align*}
    &{\ASh}_{\mathbb{Q}_{p}}\xrightarrow{\makebox[2cm]{}}{\AHS}_{\mathbb{Q}_{p}} \\
    ((A_\bullet, \lambda&, \overline{\eta}), g_i, t)\mapsto ((A_\bullet, \lambda, \overline{\eta}), g_i, g_i^{-1}t).
    \end{align*}
\end{remark}

Shadrach also considers another intermediate level structure at $p$ in the Siegel case, namely he takes $K_p$ to be the subgroup of $\text{GSp}_{2n}$ stabilizing the self dual lattice chain $\Lambda_\bullet$ and pointwise fixing the intermediate quotients $\Lambda_{i+1}/\Lambda_i$ for $i=0,\ldots, n-1$. This imposes that elements in $K_p$ are upper triangular modulo $p$ and the first $n$ entries of their diagonal belong to $1+p\Zp$: however, this only implies that the other entries belong to $k+p\Zp$ for the same $k\in \{1,\ldots,p-1\}$, which could be different from 1. Let us denote this subgroup by $\Gamma_1'(p)$: Shadrach's model in this case, denoted by ${\mathcal{A}'}_1^\text{GSp}$, is defined as follows.

\begin{definition}[\cite{Sha}, Definition 4.1]
    ${\mathcal{A}'}_1^\text{GSp}$ is the fiber product
    \[
        \begin{tikzcd}
            {{\mathcal{A}'}_1^\text{GSp}} \arrow[d] \arrow[r] & (\mathcal{G}_{OT}^\times)^n \arrow[d] \\
            {\mathcal{A}_0^\text{GSp}} \arrow[r, "\varphi"] & (OT)^n,        
        \end{tikzcd}
    \]
    where the map $\varphi$ corresponds to the groups $G_i$ for $i=0,\ldots,n-1$.
\end{definition}

\paragraph{The elliptic curve case}

Let us highlight how the various integral models relate between themselves in the case of elliptic curves. The integral model $\mathcal{A}_0^\text{GL}$ in this case parametrizes chains of degree $p^4$ isogenies
\[
A_0\xlongrightarrow{\alpha_0}A_1\xlongrightarrow{\alpha_1}A_0
\]
between four-dimensional abelian schemes plus the extra conditions of Definition \ref{def unitary Iwahori}, while $\mathcal{A}_0^\text{GSp}$ is the moduli space of degree $p$ isogenies between two elliptic curves
\[
E\xlongrightarrow{\alpha}E'
\]
plus a level-structure away from $p$, see Definition \ref{GSp iwahori integral}. These spaces are not isomorphic, but their local model are, as it is easy to deduce from their definitions (note that the isotropicity conditions in the symplectic case are vacuous for $\text{GSp}_2=\text{GL}_2$). In the deeper level situations of $\Gamma_1(p)$ and $\Gamma'_1(p)$ we have three possible integral models for the Siegel modular variety:
\begin{itemize}
    \itemsep-0em
    \item [$(1)$] $\ASh$, parameterizing an Oort-Tate generator of $\text{ker}(\alpha)$ and a $(p-1)$-th root of $p$;

    \item [$(1)$] ${\mathcal{A}'}_1^\text{GSp}$, parameterizing an Oort-Tate generator of $\text{ker}(\alpha)$;

    \item [$(1)$] $\AHS$, parameterizing an Oort-Tate generator of $\text{ker}(\alpha)$ and one of $\text{ker}(\alpha^\vee)$.
\end{itemize}

They sit in the commutative diagram
\[
\begin{tikzcd}
& \AHS \arrow[ld, swap, "\rho"] \arrow[rd] &      \\
\ASh \arrow[rr, "\theta"] & & {\mathcal{A}'}_1^\text{GSp} 
\end{tikzcd}
\]
where the map $\rho$ was defined after Remark \ref{points pro-p Iw Siegel} and the map $\theta$ forgets the $(p-1)$-th root of $p$. On the local model side the picture will turn out to be the following

\[
\begin{tikzcd}
& {\frac{B[u,v]}{\left(\substack{u^{p-1}-x\\v^{p-1}-y}\right)}} & \\
{\frac{B[u,t]}{\left(\substack{u^{p-1}-x\\t^{p-1}-p}\right)}} \arrow[ru, "t\mapsto uv"] & & {\frac{B[u]}{(u^{p-1}-x)}} \arrow[ll] \arrow[lu],
\end{tikzcd}
\]
as one can check by looking at Definition \ref{def C} and Definition \ref{def C Siegel}. Here 
\[
B=\Zp[x,y]/(xy-p)
\]
is the open neighborhood of the worst singular point in the Iwahori local model, which is the same for $\mathcal{A}_0^\text{GL}$ and
$\mathcal{A}_0^\text{GSp}$, as one easily checks from their definitions. Moreover, the local models of $\AHS$ and of $\AGL$ are isomorphic as well.	
\section{The local model diagrams}\label{Chap local models}

In this section we give a uniform interpretation of the local model constructed by Liu in \cite{Liu} and by Shadrach in \cite{Sha}, showing how both can be realized as the atlases of a (fiber product of) root stack of suitable line bundles on the corresponding local model with Iwahori level structure. For the notion of root stack and some of its properties, we refer to \cite[Section 2]{Cad}.

\subsection{Root stacks}

Let $n\geq 0$ and consider the $\Gm$-action on $\A^1$ given by $t\cdot x= t^nx$. Write $[\A^1/\Gm_{,n}]$ for the corresponding stack quotient: it classifies pairs $(L,s)$ consisting of a line bundle $L$ together with a global section of $(L^\vee)^n$; when $n=1$ these pairs are usually called {generalized Cartier divisors}. We write
\[
(\;)^n:[\A^1/\Gm]\longrightarrow [\A^1/\Gm]
\]
for the map induced by raising to the $n$-th power both on $\A^1$ and on $\Gm$. Note that it factors as 
\begin{equation}\label{factor n power map}
[\A^1/ \Gm]\xrightarrow{(n,1)} [\A^1/ \Gm_{,n}]\xrightarrow{(1,n)} [\A^1/ \Gm]
\end{equation}
where $(n,1)$ is the map induced by raising to the $n$-th power on $\A^1$ and the identity on $\Gm$ (vice versa for $(1,n)$). Let now $X$ be a scheme and $(L,s)$ a generalized Cartier divisor on $X$ corresponding to a map $X\rightarrow [\A^1/\Gm]$

\begin{definition}
    The $n$-th root stack of $X$ relative to $(L, s)$ is defined by the following pullback diagram:
    \[
    \begin{tikzcd}
        \sqrt{X} \arrow[r] \arrow[d] & \left[\A^1/ \Gm\right] \arrow[d, "(\;)^n"] \\
        X \arrow[r, "(L{,}s)"] & \left[\A^1/ \Gm\right].
    \end{tikzcd}
    \]
\end{definition}

If the generalized Cartier divisor is not clear from the situation, we will write $X\sqrt[n]{(L,s)}$ instead of $\sqrt{X}$. Taking the pullback of the maps (\ref{factor n power map}) along the atlas $\A^1\rightarrow[\A^1/\Gm]$ gives
\[
\begin{tikzcd}
{[\A^1/\mu_n]} \arrow[d] \arrow[r] & \A^1\times B\mu_n \arrow[d] \arrow[r] & \A^1 \arrow[d] \\
{[\A^1/\Gm]} \arrow[r, "(n{,}1)"] & {[\A^1/\Gm_{,n}]} \arrow[r, "(1{,}n)"]     & {[\A^1/\Gm]},  
\end{tikzcd}
\]
where $\mu_n$ acts on $\A^1$ by $\zeta\cdot t=\zeta t$. Therefore, the map $(1,n)$ is a banded $\mu_n$-gerbe and the map $(\;)^n$ is not representable. On the other hand, the map $(n,1)$ is representable by schemes. In fact, by taking the pullback
    \[
    \begin{tikzcd}
    X_1 \arrow[r] \arrow[d] & {[\A^1/\Gm]} \arrow[d, "(n{,}1)"] \\
    X \arrow[r, "(L{,}s)"] & {[\A^1/\Gm_{,n}]}      
    \end{tikzcd}
    \]
we obtain a space $X_1$ parametrizing $n$-th roots in $\Gamma(X,L)$ of $s\in \Gamma(X,(L^\vee)^{\otimes n})$: since sections of line bundles don't have automorphism, it is a scheme (in fact, it is the degree $n$ cyclic cover of $X$ determined by the generalized Cartier divisor $(L,s)$). This is in contrast with taking pullbacks along the map $(1,n)$: in this case, we are taking roots of a line bundle and such roots do admit non-trivial automorphisms. There are two natural atlases for a root stack that one can consider.
\begin{itemize}
    \item [(1)] The first one is obtained by pulling back the canonical atlas of $[\A^1/\Gm]$:
    \[
    \begin{tikzcd}
    A_1 \arrow[r] \arrow[d] & \A^1 \arrow[d]\\
    \sqrt{X} \arrow[r] & \left[ \A^1/\Gm \right].
    \end{tikzcd}
    \]
    It fits in the following diagram 
    \[
    \begin{tikzcd}[row sep = tiny, column sep = small]
        & A_1 \arrow[dd] \arrow[rr] \arrow[ld] && \A^1 \arrow[dd] \arrow[ld, "(\;)^n"] \\
        P \arrow[dd] \arrow[rr] && \A^1 \arrow[dd] & \\
        & \sqrt{X} \arrow[rr] \arrow[ld] && \left[ \A^1/\Gm \right] \arrow[ld, "(\;)^n"] \\
        X \arrow[rr, swap, "{(L,s)}"] && \left[ \A^1/\Gm \right] &   \\
    \end{tikzcd}
    \]
    where $P$ is the $\Gm$-torsor associated to $L$, and the top horizontal square is cartesian. In fact, $\sqrt{X}\cong [A_1/\Gm]$, see \cite[Proposition 2.3.5]{Cad}

    \item [(2)] The second depends on the choice of a trivializing cover $\{U_i\}_{i\in I}$ for $L$, and it is given by
    \[A_2=\coprod_{i\in I}{\text{Spec}}_{U_i}\left(\frac{\mathcal{O}_{U_i}[t]}{(t^n-s_{|U_i})}\right).
    \]
    This follows from the following isomorphism when the line bundle is trivial:
    \[
    X\sqrt[n]{(\mathcal{O}_X, s)}\cong \left[ {\text{Spec}}_{X}\left(\frac{\mathcal{O}_{X}[t]}{(t^n-s)}\right) / \mu_{n, X} \right],
    \]
    where $\mu_n$ acts by $\zeta \cdot t=\zeta t$ (see \cite[Example 2.4.1]{Cad}).
\end{itemize}

We finish this section by describing what happens when taking the root stack of a generalized Cartier divisor $(L, s)$ such that $L$ admits an $n$-th root $L'$ on $X$. This is the situation that we will find ourselves in when dealing with the local model diagram (on the Shimura side). In this case, the map $X\xrightarrow{(L,s)} [\A^1/\Gm]$ factors as $X\xrightarrow{(L',s)} [\A^1/\Gm_{,n}]\xrightarrow{(1,n)} [\A^1/\Gm]$ and we obtain the following diagram
\begin{equation}\label{section partial rootstack}
    \begin{tikzcd} [row sep = small, column sep = small]
        X_1 \arrow[rr] \arrow[dd] && \sqrt{X}\arrow[dd] \arrow[rr] & & {[\A^1/\Gm]} \arrow[dd, "{(n,1)}"] \\
        && & & \\
        X \arrow[rr, "\nu"] && X' \arrow[ddd] \arrow[rr] & & {[\A^1/\Gm_{,n}]} \arrow[ddd, "{(1,n)}"] \\
        && & & \\
        && & & \\
        && X \arrow[rr, swap, "{(L,s)}"] \arrow[rruuu, "{(L',s)}" description, dotted] &  & {[\A^1/\Gm]},
    \end{tikzcd}    
\end{equation}

where all squares are cartesian and $\nu$ is induced by $\text{id}_X$ and $(L',s)$. In particular, $\nu$ is a section of the banded $\mu_n$-gerbe $X'\rightarrow X$ which is then trivial, so we can write it as $B\mu_{n,_{X}}=X\times B\mu_n\xlongrightarrow{\text{pr}}X$ and the section $\nu$ corresponds to a map $X\longrightarrow B\mu_{n,_X}$. Now consider the pullback diagram
\[
\begin{tikzcd}
P \arrow[d] \arrow[r, "\nu'"] & X \arrow[d, "\text{atl}"] \\
X \arrow[r, "\nu"]    & {B\mu_{n,_X}},        
\end{tikzcd}
\]
where $P\rightarrow X$ is the $\mu_n$-torsor given by $\nu$. The map $\nu'$ is a torsor under the group scheme $\underline{\text{Aut}}(P)$, which is a form of $\mu_n$, hence it is in particular a $\mu_n$-torsor: assuming that $n$ is invertible on $X$, we obtain that $\nu'$, and hence $\nu$, is smooth. In the sequel we will consider the situation in which $X$ is a $\Zp$ scheme and $n=p-1$, which is invertible.

\subsection{Local model diagrams}

We can now explain the construction of the local model diagrams for the integral models that we are considering, and how it relates to root stacks. The key observation is that we can express the map $\mathcal{G}_{OT}^\times\rightarrow OT$ from the subscheme of generators to the Oort-Tate stack as a pullback:
    \[
    \begin{tikzcd}
    \mathcal{G}_{OT}^\times=\left[ \text{Spec}\left(\scalebox{1.3}{$\frac{\Zp[x,y,z]}{(xy-\omega_p, z^{p-1}-x)}$}\right)/\Gm \right] \arrow[d] \arrow[r, "z\mapsfrom s"] & {\big[ \text{Spec}\left(\Zp[s]\right) / \Gm \big]} \arrow[d, "{(p-1,1)}"] \\
    OT=\left[ \text{Spec}\left(\scalebox{1.3}{$\frac{\Zp[x,y]}{(xy-\omega_p)}$}\right)/\Gm \right] \arrow[r, "x\mapsfrom s"] & {\left[ \text{Spec}\left(\Zp[s]\right) / \Gm_{,p-1} \right].}           
    \end{tikzcd}
    \]
In particular, we can express the integral models $\ASh$ and $\AShGL$ as the following pullbacks
\[
\begin{tikzcd}
    \ASh \arrow[r] \arrow[d] & {[\A^1/\Gm]^n} \arrow[d, "{(p-1,1)^n}"] &  & \AShGL \arrow[d] \arrow[r] & {[\A^1/\Gm]^n} \arrow[d, "{(p-1,1)^n}"] \\
    {\mathcal{A}_0^\text{GSp}[t]} \arrow[r] & {[\A^1/\Gm_{,p-1}]^n} & & \mathcal{A}_0^\text{GL} \arrow[r] & {[\A^1/\Gm_{,p-1}]^n}.    
\end{tikzcd}
\]
The bottom horizontal arrow of the above diagram in the unitary case is given by the pairs $(\mathcal{L}_i,a_i)$ coming from the group schemes $G_i$, while in the Siegel case we first project to $\ASh$ and then proceed similarly. For the Haines-Stroh model we have the cartesian diagram
\[
\begin{tikzcd}
    HS=\left[ \text{Spec}\left(\scalebox{1.3}{$\frac{\Z_p[x_i, y_i, u_i, v_i]_{i=0}^{n-1}}{\left(\substack{\mathstrut x_iy_i-\omega_p, \;u_0v_0-u_iv_i \\ u_i^{p-1}-x_i,\; v_i^{p-1}-y_i}\right)_i}$}\right) \Big/ \mathbb{G}_m^n \right] \arrow[d] \arrow[r, "u_i\mapsfrom s_i", "v_i\mapsfrom t_i"'] & {\left[\text{Spec}\left(\scalebox{1.3}{$\frac{\mathbb{Z}_p[s_i, t_i]_{i=0}^{n-1}}{(s_0t_0-s_it_i)}$}\right)/\mathbb{G}_m^n\right]} \arrow[d, "(p-1{,}1)^n"] \\
    (OT)^n=\left[ \text{Spec}\left(\scalebox{1.3}{$\frac{\Z_p[x_i, y_i]_{i=0}^{n-1}}{(x_iy_i-\omega_p)_i}$}\right) / \mathbb{G}_m^n \right] \arrow[r, "x_i\mapsfrom s_i", "y_i\mapsfrom t_i"'] & {\bigg[\text{Spec}\left(\mathbb{Z}_p[s_i, t_i]_{i=0}^{n-1}\right)/\mathbb{G}_m^n\bigg]},
\end{tikzcd}
\]
where here $(p-1,1)^n$ means raising to the $(p-1)$-th power the variables $s_i,t_i$ and doing nothing on the $\Gm$ side. Thus we can write $\AHS$ as the pullback
\begin{equation}\label{HS pullback}
\begin{tikzcd}
    \AHS \arrow[d] \arrow[r] & {[C_{HS}/{\mathbb{G}_m^n}_{(1,-1)}]} \arrow[d, "{(p-1,1)^n}"] \\
    \mathcal{A}^\text{GSp}_0 \arrow[r]  & {[\A^2/\Gm_{(p-1,1-p)}]^n},
\end{tikzcd}   
\end{equation}

where
\[
C_{HS}=\text{Spec}\left(\frac{\mathbb{Z}_p[s_i, t_i]_{i=0}^{n-1}}{(s_0t_0-s_it_i)}\right)
\]
and now the bottom horizontal arrow corresponds to the triples $(\mathcal{L}_i,a_i,b_i)$. Therefore we can interpret the integral models $\mathcal{A}_1$ as fiber products of partial root stacks of the Oort-Tate line bundles on $\mathcal{A}_0$ (the Haines-Stroh model is a closed subscheme of such a fiber product). In order to construct a local model diagram for these integral models, one would like to find line triples $(\mathcal{L}_i^\text{loc},a_i^\text{loc},b_i^\text{loc})$ on $\text{M}_0$ corresponding to $(\mathcal{L}_i,a_i,b_i)$ on $\mathcal{A}_0$ under the local model diagram (i.e. they should pull back to the same triple on $\tilde{\mathcal{A}_0}$). Then, by taking pullbacks we would obtain a local model diagram for $\mathcal{A}_1$. However, this is not possible: we can only find line bundles on $\text{M}_0$ corresponding to the $(p-1)$-th power of the Oort-Tate line bundles on the integral models. Therefore, on the local model side we will have to take full root stacks, but it is still possible to relate the two sides and get a local model diagram. We start from the Siegel case, where the situation is slightly simpler since the group schemes $G_i$ are kernels of isogenies between abelian schemes, whereas in the unitary case one has to pass to the $p$-divisible groups of the abelian schemes and use Morita equivalence.

\subsubsection{The Siegel case}

\begin{notation}For a map between locally free sheaves $F\xrightarrow{\varphi}G$ of rank $r$ we will write $\bigwedge F:=\bigwedge^rF$ and $\text{det}\,\varphi$ for the global section of $(\bigwedge F)^{-1}\otimes \bigwedge G $ for corresponding to the map $\bigwedge F \rightarrow \bigwedge G$ induced by $\varphi$. If $F$ and $G$ are line bundles, we will identify the map $\varphi$ with $\text{det}\,\varphi$.
\end{notation}

Let $G=\text{Spec}_S(D)$ be a finite locally free group scheme of order $p$ over a $\Zp$-scheme $S$, with $D$ its Hopf algebra, and assume that it fits into a sequence of the form

\begin{equation*} 
0\rightarrow G\rightarrow A\xrightarrow{f} B\rightarrow 0,
\end{equation*}

where $f$ is an isogeny of abelian schemes: write $\omega_f:\omega_B\rightarrow \omega_A$ for the map induced by $f$ on the Hodge bundles. Write $(L,a,b)$ for the Oort-Tate parameters of $G$ and set $L'=L^{-1}$: recall, as in Remark \ref{OT decomposition}, that the augmentation ideal of $G$ decomposes as $I=\bigoplus_{i=1}^{p-1}(L')^{i}$ and that $a:{L'}^{\otimes p}\rightarrow L'$ is given by multiplication in $D$. The next proposition is hinted in \cite[Remark 5.6]{Liu}.

\begin{proposition}\label{fitting}
    The pair $(L^{\otimes p-1},a)$ is isomorphic to the pair $\big((\bigwedge\omega_B)^{-1}\otimes \bigwedge \omega_A, \emph{det}\,\omega_f\big)$.
\end{proposition}
\begin{proof}
First, we have the short exact sequence
\[
0\rightarrow \omega_B\longrightarrow\omega_A\longrightarrow\omega_G\rightarrow 0,
\]
obtained by one of the standard exact sequence of Kähler differentials applied to the maps $A\xrightarrow{f} B\rightarrow S$, together with the fact that $\omega_G\cong e_A^*(\Omega^1_{A/B})$ where $e_A:S\rightarrow A$ is the unit. We also have, from the other standard exact sequence for differentials, that $\omega_G\cong I/I^2$ and we claim that $I^2=a(L')\oplus (L')^2\oplus\cdots\oplus (L')^{p-1}$: this follows from the fact that $I^2=\sum_{i=2}^{2(p-1)}(L')^i$, the fact that for $i=2,\ldots, p-1$ it holds that $(L')^{i+(p-1)}\subset (L')^i$  and since $(L')^p=a({L'}^{\otimes p})$. Therefore, we have another presentation of $\omega_G$ by projective modules, namely
\[
0\rightarrow (L')^{\otimes p}\xlongrightarrow{a} L'\longrightarrow \omega_G\rightarrow 0.
\]
By looking at the $0$-th Fitting ideal of $\omega_G$, which is independent of the presentation of $\omega_G$, the proof is finished. 
\end{proof}

Note that the Cartier dual of $G$ fits in the sequence

\begin{equation*}
0\rightarrow G^*\rightarrow B^\vee\xrightarrow{f^\vee} A^\vee\rightarrow 0,
\end{equation*}

so, by applying the proposition to $G^*$, we see that the pair $(L^{\otimes 1-p},b)$ is isomorphic to $((\bigwedge \omega_{A^\vee})^{-1}\otimes \bigwedge \omega_{B^\vee}, \text{det}\, \omega_{f^\vee})$. In particular, we can (and from now on will) identify $(\bigwedge \omega_{B^\vee})^{-1}\otimes \bigwedge \omega_{A^\vee}$ with $(\bigwedge \omega_B)^{-1}\otimes \bigwedge \omega_A$ in such a way that we get the following commutative diagram:
\begin{equation}\label{diagram OT vs Hodge}
\begin{tikzcd}
    \mathcal{O}_S \arrow[r, " a"] \arrow[d, swap, "\det \omega_f"] & L^{p-1} \arrow[d, " b"] \\
    (\bigwedge \omega_B)^{-1}\otimes \bigwedge \omega_A \arrow[r, swap, "\det \omega_{f^\vee}"] \arrow[ru, "\cong" description] & \mathcal{O}_S .
\end{tikzcd}
\end{equation}

\begin{remark}
In \cite[Corollary 5.5]{Liu}, Liu proves that there exist canonical isomorphisms  
\[
[{L'}^{p}\xrightarrow{a} {L'}]\cong [\omega_{B}\xrightarrow{\omega_f} \omega_{A}] \quad \text{and} \quad [({L})^p\xrightarrow{b} {L}]\cong [\omega_{A^\vee}\xrightarrow{\omega_{f^\vee}} \omega_{B^\vee}],
\]
in the bounded derived category of coherent $\mathcal{O}_S$-modules $\text{D}^b(S)$. Taking the determinant of these maps we obtain a canonical version of the above proposition.
\end{remark}

We can apply this to $\mathcal{A}_0^\text{GSp}$. Recall that to a point $(A_\bullet, \lambda, \overline{\eta})$ of it we can associate the chain of Hodge bundles  
\[
\omega_{A_n}\longrightarrow \omega_{A_{n-}1}\longrightarrow\cdots\longrightarrow \omega_{A_0}
\]
and that we write $(\mathcal{L}_i,a_i,b_i)$ for the Oort-Tate parameters of the groups $G_i$ coming from $A_\bullet$. For each $i=0,\ldots, n-1$, let us write:
\begin{itemize}
    \itemsep0em
    \item[$(a)$] $\omega_i=(\bigwedge\omega_{A_{i+1}})^{-1}\otimes \bigwedge\omega_{A_{i}}$, a line bundle on $\mathcal{A}_0^\text{GSp}$;
    \item[$(b)$] $\det_i=\det (\omega_{A_{i+1}}\rightarrow \omega_{A_i})$, a global section of $\omega_i$;
    \item[$(c)$] $\overline{\det}_i=\det (\omega_{A^\vee_{i}}\rightarrow \omega_{A^\vee_{i+1}})$, a global section of $\omega_i^\vee$ (under the identification of $\omega_i$ with $(\bigwedge\omega_{A^\vee_{i+1}})^{-1}\otimes \bigwedge\omega_{A^\vee_{i}}$).
\end{itemize}
The diagram (\ref{diagram OT vs Hodge}) implies that the following triangle is 2-commutative 
\[
\begin{tikzcd}
&  & {[\A^2/\Gm_{(p-1,1-p)}]^n} \arrow[d] \\
\mathcal{A}_0^{\text{GSp}} \arrow[rr, "{(\omega^\vee_\bullet,\det_\bullet, \overline{\det}_\bullet)}"'] \arrow[rru, "{(\mathcal{L}^\vee_\bullet, a_\bullet, b_\bullet)}"] &  & {[\A^2/\Gm_{(1,-1)}]^n}. 
\end{tikzcd}
\]
Now, points of $\text{M}_0^\text{GSp}$ are chains of direct summands of $R^n$
\[
F_0\longrightarrow F_1\longrightarrow\cdots\longrightarrow F_n,
\]
with $F_i$ corresponding to $\omega_{A_{n-i}}$, see Subsection \ref{subsection Iwahori GSp}. Therefore, we can find triples $(\omega_i^\text{loc},\text{det}^\text{loc}_i, \overline{\text{det}}^\text{loc}_i)$ on $\text{M}_0^\text{GSp}$ whose pullback to $\tilde{\mathcal{A}}_0$ is the same as the one of $(\omega_i,\det_i, \overline{\det}_i)$. Namely, for each $i=0,\ldots,n-1$ we set:
\begin{itemize}
    \itemsep0em
    \item[$(a)$] $\omega_i^\text{loc}=(\bigwedge\,F_{n-(i+1)})^{-1}\otimes \bigwedge\,F_{n-i}$;
    \item[$(b)$] $\text{det}_i^\text{loc}=\text{det}(F_{n-(i+1)}\rightarrow F_{n-i})$, \; $\overline{\det}_i^\text{loc}=\det(\overline{F}_{n-i}^\vee\rightarrow\overline{F}_{n-(i+1)}^\vee)$.
\end{itemize}
Note that the map $F_i\rightarrow F_{i+1}$ on $\text{M}_0^\text{GSp}$ corresponds to $\omega_{A_{i+1}}\rightarrow\omega_{A_i}$. Passing to the quotient $\overline{F}_i=R^n/F_i$ and then to the $R$-linear dual, we obtain that $\overline{F}_{i+1}^\vee\rightarrow \overline{F}_{i}^\vee$ corresponds to $\omega_{A_{n-(i+1)}^\vee}\rightarrow \omega_{A_{n-i}^\vee}$: this explains the definition of $\overline{\text{det}}_i^\text{loc}$. Here we are using the following isomorphism of short exact sequences for an abelian scheme $A$ over a base $S$
\[
\begin{tikzcd}
0 \arrow[r] & \omega_{A^\vee} \arrow[r] \arrow[d, "\cong"] & H^1_{dR}(A^\vee) \arrow[r] \arrow[d, "\cong"] & \text{Lie}(A) \arrow[r] \arrow[d, "\cong"] & 0 \\
0 \arrow[r] & \text{Lie}(A^\vee)^\vee \arrow[r]  & H^1_{dR}(A)^\vee \arrow[r] & \omega_A^\vee \arrow[r] & 0,
\end{tikzcd}
\]
see \cite[Proposition 5.1.10]{BBM}. For each $i=0,\ldots, n-1$, we write $\theta_i:\mathcal{L}_i^{p-1}\rightarrow \omega_i$ for a fixed isomorphism: for example, we could choose Liu's one.


\paragraph{Shadrach's model}

By Proposition \ref{fitting} we are in the situation of diagram (\ref{section partial rootstack}) so we get the black part of the diagram below, the gray part is obtained by taking pullbacks:

\[
\begin{tikzcd}[row sep = small, column sep = small]
        &  & & & & \mathcal{A}_1 \arrow[dd] \arrow[dr] & &  & \\
    \textcolor{gray}{[\mathbb{A}^1/\mathbb{G}_m]^n} \arrow[dddd, gray, swap, "(\;)^{p-1}"] &  & \textcolor{gray}{\sqrt{\text{M}^\text{loc}_0[t]}} \arrow[dddd, gray] \arrow[ll, gray] &  & \textcolor{gray}{\sqrt{\tilde{\mathcal{A}_0}[t]}} \arrow[dddd, gray] \arrow[ll, gray] \arrow[rr,gray, crossing over] & & \sqrt{\mathcal{A}_0[t]} \arrow[dd] \arrow[rr] &  & {[\mathbb{A}^1/\mathbb{G}_m]^n} \arrow[dd, "(p-1{,}1)"] \\
    &  & & & & \mathcal{A}_0[t] \arrow[rd]& &  & \\
    &  & & & & & \mathcal{A}_0[t]' \arrow[dd] \arrow[rr] &  & {[\mathbb{A}^1/\mathbb{G}_{m, p-1}]^n} \arrow[dd, "(1{,}p-1)"] \\
    &  & & & & & &  &  \\
    \textcolor{gray}{[\mathbb{A}^1/\mathbb{G}_m]^n} &  & \textcolor{gray}{\text{M}_0[t]} \arrow[ll,gray, "{(\omega^{\text{loc},\vee}_\bullet, {\det}^\text{loc}_\bullet)}"] & & \textcolor{gray}{\tilde{\mathcal{A}_0}[t]} \arrow[ll, gray, swap, "{\phi[t]}"] \arrow[rr, gray,"{\psi[t]}"] & & {\mathcal{A}_0[t]} \arrow[rruu, dotted, "{(\mathcal{L}^\vee_\bullet, a_\bullet)}" description] \arrow[rr, swap, "{(\omega^\vee_\bullet, {\det}_\bullet)}"] &  & {[\mathbb{A}^1/\mathbb{G}_m]^n}.   
\end{tikzcd}
\]
Here $\text{M}_0=\text{M}_0^\text{GSp}$, $\phi[t]$ and $\psi[t]$ are the base change of $\phi$ and $\psi$ to $\text{M}_0[t]$, $\tilde{\mathcal{A}}_0[t]$ and $\mathcal{A}_0[t]$. We can now pull back the atlases $\A^n\rightarrow [\A^1/\Gm]^n$ on both sides to obtain

\[
\begin{tikzcd}[row sep=small, column sep=small]
    &  & & \tilde{\mathcal{A}}_1 \arrow[gray, rr] \arrow[gray, rd] & & \mathcal{B}_1 \arrow[rd] \arrow[gray, dd] & &  & \\
    \A^n \arrow[dd]  &  & \text{M}_1 \arrow[ll] \arrow[dd] & & \tilde{\mathcal{B}}_0 \arrow[gray, ll] \arrow[dd] \arrow[rr, crossing over]  & & \mathcal{B}_0 \arrow[rr] \arrow[dd]  &  &  \A^n \arrow[dd] \\
    &  & & & & \mathcal{A}_1 \arrow[dd] \arrow[dr] & &  & \\
    {[\mathbb{A}^1/\mathbb{G}_m]^n} \arrow[dddd, swap, "(\;)^{p-1}"] &  & \sqrt{\text{M}_0[t]} \arrow[dddd] \arrow[ll] &  & \sqrt{\tilde{\mathcal{A}_0}[t]} \arrow[dddd] \arrow[ll] \arrow[rr, crossing over] & & \sqrt{\mathcal{A}_0[t]} \arrow[dd] \arrow[rr] &  & {[\mathbb{A}^1/\mathbb{G}_m]^n} \arrow[dd, "(p-1{,}1)"] \\
    &  & & & & \mathcal{A}_0[t] \arrow[rd]& &  & \\
    &  & & & & & \mathcal{A}_0[t]' \arrow[dd] \arrow[rr] &  & {[\mathbb{A}^1/\mathbb{G}_{m, p-1}]^n} \arrow[dd, "(1{,}p-1)"] \\
    &  & & & & & &  &  \\
    {[\mathbb{A}^1/\mathbb{G}_m]^n} &  & {\text{M}_0[t]} \arrow[ll, "{(\omega^{\text{loc},\vee}_\bullet, {\det}_\bullet^\text{loc})}"] & & {\tilde{\mathcal{A}_0}[t]} \arrow[ll, swap, "{\phi[t]}"] \arrow[rr, "{\psi[t]}"] & & {\mathcal{A}_0[t]} \arrow[rruu, dotted, "{(\mathcal{L}^\vee_\bullet, a_\bullet)}" description] \arrow[rr, swap, "{(\omega^\vee_\bullet, {\det}_\bullet)}"] &  & {[\mathbb{A}^1/\mathbb{G}_m]^n}.           
\end{tikzcd}
\]
All the gray arrows are pullbacks of smooth maps (either $\phi[t]$, $\psi[t]$, the atlas of $[\A^1/\Gm]^n$ or the map $\mathcal{A}_0[t]\rightarrow \mathcal{A}_0[t]'$, by the discussion after diagram (\ref{section partial rootstack})), hence they are smooth. Therefore, their composition yields a local model diagram 
\[
    \begin{tikzcd}
        & \tilde{\mathcal{A}}_1 \arrow[dr, "\psi_1"] \arrow[dl, swap, "\phi_1"] &\\
        \text{M}_1 & &{\mathcal{A}}_1
    \end{tikzcd}
\]
which, for a $\Zp$-scheme $S$, can be described as follows.

\begin{itemize}
    \item $\text{M}_1(S)$ parametrizes tuples $(\beta, (f_i, \delta_i)_{i=0}^{n-1})$, where:
    \begin{itemize}
    \item[] $\beta\in \text{M}^\text{loc}_0[t](S)$;
    \item[] $f_i\in \Gamma(S,\mathcal{O}_S)$;
    \item[] $\delta_i:\omega_{i|S}^\text{loc}\xrightarrow{\cong}\mathcal{O}_S,\, \text{an isomorphism such that} \,\,
    \overline{\text{det}}^\text{loc}_{i|S}=f_i^{p-1}\delta_i$.
    \end{itemize} 

    \item $\tilde{\mathcal{A}}_1(S)$ parametrizes tuples $(\tilde{\alpha}, (\sigma_i, \gamma_i)_{i=0}^{n-1})$, where:
    \begin{itemize}
    \item[] $\tilde{\alpha}\in\tilde{\mathcal{A}}_0[t](S)$;
    \item[] $\sigma_i:\mathcal{L}_{i|S}\rightarrow \mathcal{O}_S \,\,\text{such that} \,\,\sigma_i^{p-1}=a_{i|S}$;
    \item[] $\gamma_i:\mathcal{L}_{i|S}\rightarrow \mathcal{O}_S, \,\text{an isomorphism}$.
    \end{itemize} 

    \item ${\mathcal{A}}_1$ parametrizes tuples $({\alpha}, (\sigma_i)_{i=0}^{n-1})$, where:
    \begin{itemize}
    \item[] ${\alpha}\in{\mathcal{A}}_0[t](S)$;
    \item[] $\sigma_i:\mathcal{L}_{i|S}\rightarrow \mathcal{O}_S \,\, \text{such that} \,\, \sigma_i^{p-1}=a_{i|S}$.
    \end{itemize} 
\end{itemize}

The two maps are then:
\begin{align*}
    &\phi_1:(\tilde{\alpha}, \sigma_i, \gamma_i)_i\longmapsto (\phi[t]\tilde{\alpha}, \, \sigma_i\gamma_i^{-1}, \, \gamma_i^{p-1}\theta_{i|S}); \\
    &\psi_1:(\tilde{\alpha}, \sigma_i, \gamma_i)_i\longmapsto (\psi[t]\tilde{\alpha},\, \sigma_i). 
\end{align*}

\begin{remark}\label{change of atlas for local model}
    \begin{itemize}
        \itemsep0em
        \item [$(i)$] The local model constructed in \cite[Theorem 4.4]{Sha} is the atlas of type (2) of $\sqrt{\text{M}_0[t]}$ relative to the covering defined at the end of Subsection 1.2.1. In order to construct Shadrach's local model diagram in the same way as we did above, we would have to find a cover of $\mathcal{A}_0[t]$ which trivializes the $\omega_i$ simultaneously, and having the same preimage along $\psi[t]$ as the covering of the local model. We do not know how to achieve this, since the opens $U_x$ are not $\mathcal{G}$-invariant so we cannot just take $\phi^{-1}\psi(U_x)$. Nevertheless, the two local models are related by a smooth correspondence, being atlases of the same algebraic stack, so we can use any of the two in order to prove étale-local geometric properties of $\ASh$.
        \item [$(ii)$] Since the $\mathcal{G}$-translates of the open $U=U_\tau$ corresponding to the unique length zero $\mu$-admissible element cover $\text{M}_0$, we can restrict directly to $U$ in order to find a local model for $\ASh$, and in fact this is what Shadrach does. Using his notation, we call $U_1$ the local model obtained in this way.
    \end{itemize}
    
\end{remark}

\paragraph{Haines-Stroh's model}

The construction of the local model diagram in the Haines-Stroh's case is essentially the same: the difference is that we want the Oort-Tate parameters to satisfy the relations $g_0g_0^*=g_ig_i^*$ for $i=0,\ldots,n-1$, hence we have to work with $C_{HS}=\text{Spec}\,\left(\frac{\Zp[s_i,t_i]_{i=0}^{n-1}}{(s_0t_0-s_it_i)_i}\right)$. The diagram we obtain is the following, where we write $\omega_\bullet$ instead of $(\omega^\vee_\bullet, \det_\bullet, \overline{\det}_\bullet)$ (similarly for $\omega_\bullet^\text{loc}$ and $\mathcal{L}_\bullet)$ for ease of notation. 

\[
\begin{tikzcd} [row sep=small, column sep=small]
    &  & & \tilde{\mathcal{A}}_1 \arrow[rr, gray] \arrow[rd, gray] & & \mathcal{B}_1 \arrow[rd] \arrow[dd, gray] &&  &\\
    C_{HS} \arrow[dd] &  & \text{M}_1 \arrow[ll] \arrow[dd]& & \tilde{\mathcal{B}}_0 \arrow[ll, gray] \arrow[dd] \arrow[rr, crossing over] && \mathcal{B}_0 \arrow[dd] \arrow[rr]  &  & C_{HS} \arrow[dd] \\
    &  & & & & \mathcal{A}_1 \arrow[dd] \arrow[dr] & &  & \\
    {[C_{HS}/{\mathbb{G}^n_m}_{(1,-1)}]} \arrow[dddd, swap, "(\;)^{p-1}"] &  & \sqrt{\text{M}_0} \arrow[dddd] \arrow[ll] & & \sqrt{\tilde{\mathcal{A}_0}} \arrow[dddd] \arrow[ll] \arrow[rr, crossing over] & & \sqrt{\mathcal{A}_0} \arrow[dd] \arrow[rr] &  & {[C_{HS}/{\mathbb{G}_m^n}_{(1,-1)}]} \arrow[dd, "{(p-1,1)^n}"] \\
    &  & &  & & \mathcal{A}_0 \arrow[rd] & &  & \\
    &  & & & & & \mathcal{A}_0' \arrow[dd] \arrow[rr] &  & {[\A^2/\Gm_{(p-1,1-p)}]^n} \arrow[dd, "{(1,p-1)^n}"] \\
    &  & & & & & &  &  \\
    {[\mathbb{A}^2/\mathbb{G}_{m, (1,-1)}]^n} &  & \text{M}_0 \arrow[ll, "\omega_\bullet^\text{loc}"] & & \tilde{\mathcal{A}_0} \arrow[ll,swap, "\phi"] \arrow[rr, "\psi"] & & \mathcal{A}_0 \arrow[rruu, dotted, "{\mathcal{L}_\bullet}" description] \arrow[rr, swap, "\omega_\bullet"]  &  & {[\A^2/\Gm_{(1,-1)}]^n},     
\end{tikzcd}
\]

As in Shadrach's case, the composition of the (smooth) gray arrows yields a local model diagram
\[
    \begin{tikzcd}
        & \tilde{\mathcal{A}}_1 \arrow[dr, "\psi_1"] \arrow[dl, swap, "\phi_1"] &\\
        \text{M}_1 & &{\mathcal{A}}_1
    \end{tikzcd}
\]
which, for a $\Zp$-scheme $S$, admit the following moduli description.

\begin{itemize}
    \item $\text{M}_1(S)$ parametrizes tuples $(\beta, (f_i, g_i, \delta_i)_{i=0}^{n-1})$, where:
    \begin{itemize}
    \item[] $\beta\in \text{M}^\text{loc}_0(S)$;
    \item[] $f_i, g_i\in \Gamma(S,\mathcal{O}_S), \, \text{such that} \,\, f_0g_0=f_ig_i \,\, \forall i$;
    \item[] $\delta_i:\omega_{i|S}^\text{loc}\xrightarrow{\cong}\mathcal{O}_S,\, \text{an isomorphism such that}\,\,
    \overline{\text{det}}^\text{loc}_{i|S}=\delta_i^{-1}f_i^{p-1} \, \text{and} \, \text{det}_{i|S}^\text{loc}=g_i^{p-1}\delta_i$.
    \end{itemize}

    \item $\tilde{\mathcal{A}}_1(S)$ parametrizes tuples $(\tilde{\alpha}, (\sigma_i, \varepsilon_i, \gamma_i)_{i=0}^{n-1})$, where:
    \begin{itemize}
    \item[] $\tilde{\alpha}\in\tilde{\mathcal{A}}_0(S)$;
    \item[] $\mathcal{O}_S\xrightarrow{\varepsilon_i}\mathcal{L}_{i|S}\xrightarrow{\sigma_i} \mathcal{O}_S \,\, \text{such that} \,\, \sigma_i^{p-1}=a_i,\,\,\varepsilon_i^{p-1}=b_i, \,  \,\sigma_i\varepsilon_i=\sigma_0\varepsilon_0$;
    \item[] $\gamma_i:\mathcal{L}_{i|S}\rightarrow \mathcal{O}_S,\, \text{an isomorphism}$.
    \end{itemize} 

    \item ${\mathcal{A}}_1$ parametrizes tuples $({\alpha}, (\sigma_i, \varepsilon_i)_{i=0}^{n-1})$, where:
    \begin{itemize}
    \item[] ${\alpha}\in{\mathcal{A}}_0(S)$;
    \item[] $\mathcal{O}_S\xrightarrow{\varepsilon_i}\mathcal{L}_{i|S}\xrightarrow{\sigma_i} \mathcal{O}_S \,\, \text{such that} \,\, \sigma_i^{p-1}=a_i, \,\,\varepsilon_i^{p-1}=b_i, \, \, \sigma_i\varepsilon_i=\sigma_0\varepsilon_0$.
    \end{itemize} 
\end{itemize}

The two maps are then:
\begin{align*}
    &\phi_1:(\tilde{\alpha}, \sigma_i, \varepsilon_i, \gamma_i)_i\longmapsto (\phi\tilde{\alpha}, \, \gamma_i\varepsilon_i, \,\sigma_i\gamma_i^{-1}, \, \gamma_i^{p-1}\theta_{i|S}); \\
    &\psi_1:(\tilde{\alpha}, \sigma_i, \varepsilon_i, \gamma_i)_i\longmapsto (\psi\tilde{\alpha},\, \varepsilon_i,\, \sigma_i). 
\end{align*}

\begin{remark}
    \begin{itemize}
        \item [$(i)$] One can check that the local model diagram we just described agrees with the one constructed by Liu in \cite[Definition 5.8]{Liu} (in the Siegel case).
        \item [$(ii)$] Remark \ref{change of atlas for local model} applies in this case as well and in fact, in Section \ref{HS non norm nor flat}, we are going to use an atlas of type $(1)$ as a local model to show that $\AHS$ is neither flat over $\Zp$ nor normal.
    \end{itemize}
\end{remark}
        
\subsubsection{The unitary case}

Here the situation is slightly more complicated, since the group schemes $G_i$ appearing in the definition of $\AShGL$ are not the kernels of isogenies between abelian schemes. Nevertheless, we can still get the following diagram
\[
\begin{tikzcd}[row sep=small, column sep=small]
    &  & & \tilde{\mathcal{A}}_1 \arrow[gray, rr] \arrow[gray, rd] & & \mathcal{B}_1 \arrow[rd] \arrow[gray, dd] & &  & \\
     \A^n \arrow[dd]  &  & \text{M}_1 \arrow[ll] \arrow[dd] & & \tilde{\mathcal{B}}_0 \arrow[gray, ll] \arrow[dd] \arrow[rr, crossing over]  & & \mathcal{B}_0 \arrow[rr] \arrow[dd]  &  &  \A^n \arrow[dd] \\
    &  & & & & \mathcal{A}_1 \arrow[dd] \arrow[dr] & &  & \\
    {[\mathbb{A}^1/\mathbb{G}_m]^n} \arrow[dddd, swap, "(\;)^{p-1}"] &  & \sqrt{\text{M}_0} \arrow[dddd] \arrow[ll] &  & \sqrt{\tilde{\mathcal{A}_0}} \arrow[dddd] \arrow[ll] \arrow[rr, crossing over] & & \sqrt{\mathcal{A}_0} \arrow[dd] \arrow[rr] &  & {[\mathbb{A}^1/\mathbb{G}_m]^n} \arrow[dd, "(p-1{,}1)"] \\
    &  & & & & \mathcal{A}_0 \arrow[rd]& &  & \\
    &  & & & & & \mathcal{A}_0' \arrow[dd] \arrow[rr] &  & {[\mathbb{A}^1/\mathbb{G}_{m, p-1}]^n} \arrow[dd, "(1{,}p-1)"] \\
    &  & & & & & &  &  \\
    {[\mathbb{A}^1/\mathbb{G}_m]^n} &  & {\text{M}_0} \arrow[ll, "{(\omega^{\text{loc},\vee}_\bullet, {\det}_\bullet^\text{loc})}"] & & {\tilde{\mathcal{A}_0}} \arrow[ll, swap, "{\phi}"] \arrow[rr, "{\psi}"] & & {\mathcal{A}_0} \arrow[rruu, dotted, "{(\mathcal{L}^\vee_\bullet, a_\bullet)}" description] \arrow[rr, swap, "{(\omega^\vee_\bullet, {\det}_\bullet)}"] &  & {[\mathbb{A}^1/\mathbb{G}_m]^n}        
\end{tikzcd}
\]
in a similar way as before: we explain how. Suppose we are given to abelian schemes $A$ and $B$ with an action of $\mathcal{O}_B\otimes \Z_{(p)}$ and a degree $p^{2n}$ isogeny between them, compatible with the action, as in Definition \ref{def unitary Iwahori}:
\[
0\rightarrow H\longrightarrow A\xlongrightarrow{f} B\rightarrow 0.
\]
In this case, $H$ is locally free of rank $p^{2n}$ and admits an action of $\text{M}_n(\Zp)\times \text{M}_n(\Zp)$, so we get a locally free group scheme $G=\text{e}_{11}H$ of rank $p$ as in (\ref{Gi}): let $(L,a,b)$ be its Oort-Tate parameters. The short exact sequence of Hodge bundles induced by $f$ also admits an action by $\text{M}_n(\Zp)\times \text{M}_n(\Zp)$ and applying the idempotent $\text{e}_{11}$ we obtain
\[
0\rightarrow \text{e}_{11}\omega_B\xlongrightarrow{\text{e}_{11}\omega_f} \text{e}_{11}\omega_A\longrightarrow \text{e}_{11}\omega_H\rightarrow 0.
\]
Now, since $\text{e}_{11}\omega_H\cong\omega_{\text{e}_{11}H}$, we can apply Proposition \ref{fitting} to deduce that $(L^{p-1},a)$ is isomorphic to $((\bigwedge \text{e}_{11}\omega_B)^{-1}\otimes \bigwedge \text{e}_{11}\omega_A, \text{det}\, \text{e}_{11}\omega_f )$. Since the local model diagram in the unitary case is constructed by considering $\text{e}_{11}\omega_{A_i}$, we can construct the $\Gamma_1(p)$-local model diagram as in the Siegel case. It admits a description completely analogous to the one in the Siegel case, so we do not repeat it here.

\section{Normality and flatness}\label{chap norm flat}

In this section we prove that the integral models $\AShGL$ and $\AHS$ are not normal, and that Haines-Stroh model is not flat over $\Zp$: we do this by explicit computation on their local models. More precisely, in the unitary case we give a description of the irreducible components of the local model and deduce that it is not normal (except in a Drinfeld case, where the local model is regular). In the Siegel case, we show that the local model of the Haines-Stroh model is neither flat nor normal, except for the case of elliptic curves. In the unitary case, we also give a presentation of the normalization of Shadrach's local model in one of the Drinfeld cases.

\subsection{Shadrach's unitary model}

Recall from Subsection \ref{subsection Iwahori level} that to any admissible alcove $x$ we can associate an open affine subscheme of $\text{M}_0^\text{GL}$ 
\[
U_x(R)=\{ F_\bullet \in \text{M}^\text{GL}_{0}(R) \;|\; F_i\oplus ( \bigoplus_{t_i^x(j)=0} R\text{e}_j)=R^n \}.
\]

It comes equipped with an $\Fp$-point denoted by $ F^x_\bullet$, where $F^x_i=\bigoplus_{t^x_i(j)=1}\Fp \text{e}_j $. Now let $\tau$ be the unique alcove whose corresponding stratum is a single closed point, called the worst singular point: we denote it by $\uptau$. It is given by

\begin{align*}
    &\tau_0=(1^{(r)}, 0^{(n-r)})  &  &t_0^\tau = (1^{(r)}, 0^{(n-r)})\\
    &\tau_1=(1^{(r+1)}, 0^{(n-r-1)})   &  &t_1^\tau = (0, 1^{(r)}, 0^{(n-r-1)})\\
    &\vdots & &\vdots\\ 
    &\tau_{n-1}=(2^{(r-1)}, 1^{(n-r+1)})  &  &t_{n-1}^\tau = (1^{(r-1)}, 0^{(n-r)}, 1)\\
\end{align*}

We write $U_\tau=U=\text{Spec}(B)$: it can be shown (see \cite[Section 4.4.2]{Goer}) that 
$$B=\Zp[a^i_{jk} | i=0,\ldots, n-1; j=1, \ldots, n-r; k=1, \ldots, r]/I,$$
where $I$ is the ideal generated by the following elements, for $i=0,\ldots, n-1$ (we set $a^n_{jk}=a^0_{jk}$):

\begin{align*}
    &a^{i+1}_{j,r}a^i_{11}-a^i_{j+1,1}, \hspace{0.25cm} a^{i+1}_{n-r,r}a^i_{11}-p, & j=1,\ldots, n-r-1;\\[1em]
    &a^i_{j+1,k}-(a^{i+1}_{j, k-1}+a^{i+1}_{j,r}a^i_{1k}), & j=1,\ldots, n-r-1; k=2,\ldots, r;\\[1em]
    &a^{i+1}_{n-r, k-1}+a^{i+1}_{n-r, r}a^i_{1,k}, &k=2,\ldots, r.
\end{align*}

Indeed, if $F_\bullet\in \text{M}^\text{GL}_0(R)$ is an $R$ point of $U$, each summand $F_i$ is actually a free $R$-module of rank $r$. The conditions in the definition of $U$ imply that we can choose a basis for each $F_i$ such that $F_i$ is represented by an $n\times r$ matrix $M_i$ whose $i+1,\ldots, i+r$ rows (index taken modulo $n$) are the $r\times r$ unit matrix. Therefore $F_\bullet$ can be written in the form

    \[
    \begin{tikzcd}
        \setlength\arraycolsep{1pt}
        \scalebox{0.8}{ $\begin{pmatrix}
            1 &   & \\   
            &    & \\
            &    & 1  \\
            a^0_{11}  &\cdots & a^0_{1,r}\\
            \vdots & &\vdots\\
            a^0_{n-r,1}  &\cdots & a^0_{n-r,r}
        \end{pmatrix} $} \arrow[r, "\varphi_{0 _{|F_0}}"] 
        &
        \setlength\arraycolsep{1pt}
        \scalebox{0.8}{$ \begin{pmatrix}
            a^1_{n-r,1}  &\cdots & a^1_{n-r,r}\\
            1 &   & \\   
            &    & \\
            &    & 1  \\
            a^1_{11}  &\cdots & a^1_{1,r}\\
            \vdots & &\vdots            
        \end{pmatrix} $} \arrow[r, "\varphi_{1 _{|F_1}}"]
        &
        \cdots \arrow[r, "\varphi_{n-1 _{|F_{n-1}}}"]
        &
        \setlength\arraycolsep{1pt}
        \scalebox{0.8}{$ \begin{pmatrix}
            1 &   & \\   
            &    & \\
            &    & 1  \\
            a^0_{11}  &\cdots & a^0_{1,r}\\
            \vdots & &\vdots\\
            a^0_{n-r,1}  &\cdots & a^0_{n-r,r}
        \end{pmatrix} $}.
    \end{tikzcd}
    \]

The condition that $\varphi_i$ sends $F_i$ to $F_{i+1}$ can be expressed by $\varphi_i M_i=M_{i+1}A_i$ for some $r\times r$ matrix $A_i$, which is uniquely determined by $M_i, M_{i+1}$ and equal to 
\[
\begin{pmatrix}
0 & 1 &  & \\
  &   &  & \\
  &   & 0 & 1  \\
a^i_{11} &a^i_{12} &\cdots & a^i_{1r}
\end{pmatrix}.
\]
Then one gets the above equations by comparing the entries of $\varphi_i M_i$ and $M_{i+1}A_i$. In \cite[Proposition 4.13]{Goer} it is shown that $B\cong\Zp[a^i_k | i=0,\ldots,n-1; k=1,\ldots, r]/J$, where now $J$ is the ideal generated by the entries of the matrices
\begin{equation}\label{matrices}
    A_{n-1}A_{n-2}\cdots A_0-p,\; A_{n-2}\cdots A_0A_{n-1},\;\ldots,\; A_0A_{n-1}\cdots A_1 -p.
\end{equation}
The variables $a^i_k$ in the second description correspond to the variables $a^i_{1k}$ in the first one.

\begin{remark}
    In \cite{Sha} Shadrach uses the dual of $H^1_{dR}(A_\bullet)$ to write down the local model diagram in the Iwahori case, see Subsection 2.2 of loc.cit., whereas we use $H^1_{dR}(A_\bullet)$. This explains why in the next definition we take roots of $a^{i+1}_{n-r,r}$ (which corresponds to the determinant of $\overline{F}_i\rightarrow \overline{F}_{i+1}$) instead of $a^i_{11}$, (which is the determinant of $F_i\rightarrow F_{i+1}$), as one might expect from Proposition \ref{fitting}.
\end{remark}

\begin{definition}[{\cite[Definition 3.16]{Sha}}]\label{def C}
    Shadrach's local model in the pro-$p$ Iwahori case is given by $U_1=\text{Spec}(C)$, where
    \[
    C=\frac{B[u_0,\ldots, u_{n-1}]}{(u_i^{p-1}-a^{i+1}_{n-r,r})_i}.
    \]
    If needed, we will write $C^{(r)}$ to indicate that we are in the case of signature $(n-r, r)$.
\end{definition}

For ease of notation, we set $\alpha_i=a^{i+1}_{n-r,r}$ and $\beta_i=a^i_{11}$, so $\alpha_i\beta_i=p$. Note that we are using the first description of $B$ in order to define the pro-$p$ local model. As we are going to see, the geometry of $U_1$ depends on $p$ in a more substantial way than the Iwahori case.

\begin{proposition}\label{geom structure U1}
    The local model $U_1$ is is flat over $\Zp$ and its generic fiber is normal, with $g$ connected components, where $g=\emph{gcd}(p-1, n-r)$. Moreover, $U_1$ is connected, reduced and has $g$ irreducible components.
\end{proposition}
\begin{proof}
    The local model $U_1$ is finite and free over $U$, which is flat by the main result of \cite{Goer}, so it is flat over $\Zp$. The generic fiber $U_{1,\Qp}$ is normal since it is a finite étale extension of $U_{\Qp}$, which is normal. We are now going to explicitly compute the connected components of $U_{1,\Qp}$, which coincide with its irreducible components by normality: from the flatness of $U_1$ over $\Zp$ it follows that irreducible components of $U_1$ are the closure of the connected components of $U_{1,\Qp}$. The connectedness assertion will follow from the fact that all the irreducible components intersect in one point, while the reducedness follows clearly from the flatness of $U_1$ over $\Zp$ and the fact that $U_{1,\Qp}$ is reduced.\\
    We start by noting that $(u_0\cdots u_{n-1})^{p-1}=a_{n-r,r}^0\cdots a_{n-r,r}^{n-1}=p^{n-r}$ in $C$, as $a_{n-r,r}^i$ is the restriction to $U$ of $\text{det}(\bar{\varphi}_i)$ and $\bar{\varphi}_0\cdots \bar{\varphi}_{n-1}=p^{n-r}$: therefore we get 
    \[
    \prod_{\varepsilon \in \mu_g} {{(u_0\cdots u_{n-1})^{\frac{p-1}{g}}- \varepsilon p^{\frac{n-r}{g}}}} =0,
    \]
    where $\mu_g$ is the group of $g$-th roots of unity in $\Zp$. Write $p-1=g\cdot h$, $n-r=g\cdot m$ and $f_\varepsilon=(u_0\cdots u_{n-1})^h-\varepsilon p^m$. The ideals generated by one of the $f_\varepsilon$ are coprime one to each other once $p$ is inverted, thus we obtain the decomposition
    \[
    C_{\Qp}=\prod_{\varepsilon\in \mu_g}\frac{B_{\Qp}[u_0,\ldots, u_{n-1}]}{\left( {u_i^{p-1}-a_{n-r,r}^{i+1},\: f_\varepsilon}\right)}.
    \]
    Let us write $C_\varepsilon=C/(f_\varepsilon)$, so that $C_{\Qp}=\prod_{\varepsilon\in\mu_g}C_{\varepsilon, \Qp}$. Before proceeding, we need an explicit description of $B_{\Qp}$. Let $F_\bullet \in U_{\Qp}(R)$: the condition $\varphi_i\cdots \varphi_0(F_0)\subset F_{i+1}$ can be expressed in terms of matrices similarly as in the discussion above, namely by $\varphi_i\cdots \varphi_0 M_0=M_{i+1}\tilde{A}_i$. Here $\tilde{A}_i$ is the $r\times r$ submatrix of the product
    \[
    \setlength\arraycolsep{7pt}
        \scalebox{0.8}{$ \begin{pmatrix}
            p\cdot \textbf{1}_{i+1} &\\
            &\\
            &\\
            & \textbf{1}_{n-i-1}
        \end{pmatrix} $}
        \cdot
        \setlength\arraycolsep{1pt}
        \scalebox{0.8}{$ \begin{pmatrix}
            1 &   & \\   
            &    & \\
            &    & 1  \\
            a^0_{11}  &\cdots & a^0_{1,r}\\
            \vdots & &\vdots\\
            a^0_{n-r,1}  &\cdots & a^0_{n-r,r}
        \end{pmatrix}$ }
    \]
    given by the $i+2,\ldots,i+r+1$ rows (indexed modulo $n$). Since $p$ is invertible, so is the matrix $\tilde{A}_i$ for each $i$, hence we can write the variables $a^i_{jk}$ in terms of $a^0_{jk}$ for $i=1,\ldots, n-1$. This shows that 
    \[
    B_{\Qp}\cong \Qp \left[a_{jk}^0, \scalebox{1.1}{$ \frac{1}{\text{det}(\tilde{A}_i)}$} \,|\, \substack{
            i=0,\dots,n-1 \\
            j=1,\dots,n-r \\
            k=1,\dots,r}\right].
    \]
    Note that $\tilde{A}_i$ is the matrix representing $(\varphi_i\cdots \varphi_0)_{|F_0}$ and therefore it is equal to the product $A_i\cdots A_0$: in particular, we have equalities
    \[
    a^i_{11}\cdots a^0_{11}= \text{det}({A}_i)\cdots \text{det}({A}_0)=\text{det}(\tilde{A}_i).
    \]
    This implies that, under the above isomorphism, the variable $a^i_{11}$ is mapped to $\frac{\text{det}(\tilde{A}_{i})}{\text{det}(\tilde{A}_{i-1})}$ and thus $a^{i+1}_{n-r,r}$ corresponds to $p\cdot \frac{\text{det}(\tilde{A}_{i-1})}{\text{det}(\tilde{A}_i)}$, which we again denote by $\alpha_i$. Hence we get that
    \[
    C_{\Qp}\cong \prod_{\varepsilon\in\mu_g} \frac{\Qp \left[a_{jk}, \frac{1}{\text{det}(\tilde{A}_i)}, u_i \,|\, \substack{
            i=0,\dots,n-1 \\
            j=1,\dots,n-r \\
            k=1,\dots,r}\right]}{\left({ u_i^{p-1}-\alpha_i,\; (u_0\cdots u_{n-1})^h-\varepsilon p^m}\right)_i}
    \]
    (we are suppressing the superscript $0$ and writing $a_{jk}$ for the sake of simplifying the notation). Note that any $\varepsilon\in \mu_g$ can be written as $\varepsilon=\zeta^h$ for some $\zeta\in \mu_{p-1}$, so we have $C_\varepsilon \cong C_1$ (non canonically). Moreover, sending $u_{n-1}$ to $u_{n-1}/u_0\cdots u_{n-2}$ gives an isomorphism between $C_{1, \Qp}$ and the ring 
        \[
        \frac{\Qp \left[a_{jk}, \scalebox{1}{$ \frac{1}{\text{det}(\tilde{A}_i)}$}, u_i \,|\, \substack{
            i=0,\dots,n-1 \\
            j=1,\dots,n-r \\
            k=1,\dots,r}\right]}{\left({ u_i^{p-1}-\alpha_i,\, (u_{n-1})^{h}-  p^{m} }\right)_{i\neq n-1}},
        \]
        as the equation $u_{n-1}^{p-1}-\alpha_{n-1}$ becomes $u_{n-1}^{p-1}-p^{n-r}$: this shows that $C_{1,\Qp}$ is flat over $\Qp \left[a_{jk}, \scalebox{1}{$ \frac{1}{\text{det}(\tilde{A}_i)}$} \right]$, hence it embeds into $C_{1, \Qp} \otimes K$, where $K=\text{Frac}(\Qp[a_{jk}])$.
    
    \begin{claim}
        The rings $C_{\varepsilon, \Qp}$ are normal integral domains.
    \end{claim}
    \begin{proofc}
        Without loss of generality, assume $\varepsilon=1$ for the proof of the claim. We need to show that $C_{1, \Qp} \otimes K$ is a domain, for which it is enough to show that the surjection $C_{1, \Qp} \otimes K \twoheadrightarrow K(\sqrt[p-1]{\alpha_i},\sqrt[h]{ p^{m}})_{i\neq n-1}=L$ is an isomorphism. This holds if and only if both terms have the same dimension as $K$-vector spaces, if and only if $[L:K]=h(p-1)^{n-1}$: as $[L:K]=[L:K(\sqrt[h]{ p^m})][K(\sqrt[h]{ p^m}):K]$, we compute these two degree separately starting with the first.
        \begin{itemize}
            \item Write $F=K(\sqrt[h]{ p^m})$ and ${\Lambda}$ for the subgroup of $F^{\times} / (F^\times )^{p-1}$ generated by the classes of $\alpha_i$, for $i=0,\ldots, n-2$. Now $L$ is obtained from $K(\sqrt[h]{ p^m})$  by adjoining ($p-1$)-th roots of the elements $\alpha_i$, hence by Kummer theory $[L:K(\sqrt[h]{ p^m})]= |{\Lambda}|$. There is a surjective map
            \[
                \left( \frac{\Z}{(p-1)\Z} \right)^{n-1} \twoheadrightarrow\Lambda, \quad \underline{x} \mapsto \prod_{i=0}^{n-2}\alpha_i^{x_i}
            \]
            whose kernel consists of those $\underline{x}$ such that $\prod_{i=0}^{n-2}\alpha_i^{x_i}\in (F^\times)^{p-1}$. By the definition of $\alpha_i$, a simple calculation shows that
            \begin{equation}\label{product gamma_i}
            \prod_{i=0}^{n-2}\alpha_i^{x_i}= p^{\sum_i x_i}\cdot \text{det}(\tilde{A}_0)^{x_1-x_0}\cdots \text{det}(\tilde{A}_{n-3})^{x_{n-2}-x_{n-3}}\cdot\text{det}(\tilde{A}_{n-2})^{-x_{n-2}}.
            \end{equation}
            Note that, for each $i=0,\ldots, n-2$, $\text{det}(\tilde{A}_i)$ can be computed block-wise, and it is the product of a power of $p$ and the determinant of a matrix whose entries are some of the variables $a_{jk}$: this implies that $\text{det}(\tilde{A}_i)$ is an irreducible polynomial. Therefore, for the right hand side of equation (\ref{product gamma_i}) to belong to $(F^\times)^{p-1}$, we must have that $p-1$ divides $x_{n-2}$ and $x_{i+1}-x_i$ for each $i=0,\ldots, n-3$. But the conditions $0\leq x_i\leq p-2$ then imply that $x_{n-2}=\cdots=x_0=0$, i.e. the map is bijective, which shows that $[L:K(\sqrt[h]{ p^m})]=(p-1)^{n-1}$.
            \item $[K(\sqrt[h]{ p^m}):K]=h$, as can be seen by applying Kummer theory analogously to the previous case (and using that $h$ and $m$ are coprime).
        \end{itemize}
    \end{proofc}    
     To conclude, we show that all irreducible components intersect in the preimage of the worst singular point $\uptau$ under the natural map $\rho:U_1\rightarrow U$. Since $\uptau$ corresponds to the morphism $B\rightarrow \Fp$ sending all the variables $a^i_{jk}$ to zero, we have an isomorphism $\rho^{-1}(\uptau)\cong \Fp[u_0,\ldots, u_{n-1}]/(u_i^{p-1})$ and we see that its unique topological point corresponds to the map $C\rightarrow \Fp$ sending all the variables (including $u_i$) to zero. Now, $\rho$ is finite and surjective, so each irreducible component of $U_{1,\Fp}$ maps surjectively onto an irreducible component of $U_{\Fp}$: since $\uptau$ is contained in the intersection of all irreducible components of $U_{\Fp}$, all the irreducible components of $U_{1,\Fp}$ intersect in $\rho^{-1}(\uptau)$. Then, each irreducible component of $U_1$ contains $\rho^{-1}(\uptau)$ in its special fiber and hence they all intersect.
\end{proof}

\begin{remark}\label{flatness irr components}
    \begin{itemize}
        \item [$(i)$] If $g=1$, the local model $U_1$ is irreducible. Since it is also reduced, we see that it is an integral scheme.
        
        \item [$(ii)$] Suppose that $g>1$: it would be tempting to believe that the reduced scheme structure on the irreducible component of $U_1$ corresponding to $\varepsilon\in\mu_g$ is given by the ring
        \[
        C_\varepsilon = \frac{B[u_0,\ldots, u_{n-1}]}{\left( \scalebox{1.3}{$\substack{u_i^{p-1}-\alpha_i\\ (u_0\cdots u_{n-1})^h-\varepsilon p^m}$}\right)}.
        \]
        This would be true if $C_\varepsilon$ were flat over $\Zp$, in which case it would embed inside $C_{\varepsilon, \Qp}$ (which is reduced) and it would coincide with the scheme theoretic closure of its generic fiber. Unfortunately, it is not flat over $\Zp$. Indeed, the following calculation shows that $C_1$ has non-trivial $p$-torsion:
        \begin{align*}
            &p^m\big((u_1\cdots u_{n-1})^h-\beta_0^mu_0^{m(p-1)-h}\big)=\beta_0^m\big( \alpha_0^m(u_1\cdots u_{n-1})^h-p^m u_{0}^{m(p-1)-h} \big)=\\
            &=\beta_0^m\big( u_0^{m(p-1)}(u_1\cdots u_{n-1})^h-(u_0\cdots u_{n-1})^hu_{0}^{m(p-1)-h} \big)=0.
        \end{align*}
        Note that the element $(u_1\cdots u_{n-1})^h-\beta_0^mu_0^{m(p-1)-h}$ is non-zero inside $C_1$: if it were, we would get that $(u_1\cdots u_{n-1})^h=0$ inside 
        \[
        \frac{C_1}{(u_0)}=\frac{\big(B_{\Fp}/(\alpha_0)\big)[u_1,\ldots,u_{n-1}]}{(u_i^{p-1}-\alpha_i)_{i>0}},
        \]
        but this is a free $B_{\Fp}/(\alpha_0)$-module and $(u_1\cdots u_{n-1})^h$ is part of a basis, so it cannot be zero. Nevertheless, we can prove that the rings $C_\varepsilon$ are topologically flat, i.e. the generic fiber of $\text{Spec}(C_\varepsilon)$ is Zariski-dense.
    \end{itemize}
\end{remark}

\begin{lemma}\label{top flat unitary}
    The ring $C_\varepsilon$ is topologically flat.
\end{lemma}
\begin{proof}
    We prove this for $C_1$. It is enough to prove the following: given any characteristic $p$ field $k$ and $\bar{x}\in C_1(k)$, there exists a mixed characteristic DVR $R$, with residue field $k$, and an element $x\in C_1(R)$ such that $x$ maps the closed point of $\text{Spec}(R)$ to the topological point underlying $\bar{x}$. For such an $R$-point $x$, let us write $x_k$ and $x_K$ for its restriction to $k$ and $K$ respectively, where $K=\text{Frac}(R)$. By the flatness of $C$ over $\Zp$, given $\bar{x}\in C_1(k)$ we can find an $x'\in C(R)$ satisfying the above condition for $C$. The question is whether $x$ maps the generic point of $\text{Spec}(R)$ to another connected component of $\text{Spec}(C_{\Qp})$ or not, but the next claim shows that we can suitably modify $x'$ to another $x\in C(R)$ such that $x_k=x'_k$ and $x_K\in C_1(K)$. 
    \begin{claim}
        Let $y\in C(R)$ and suppose that $y_K\in C_\varepsilon(K)$ for some $\varepsilon \in \mu_g$. Then there exists another $x\in C(R)$ such that $x_k=y_k$ and $x_K\in C_1(K)$.
    \end{claim}
    \begin{proofc}
        The point $y$ is determined by where it maps the variables $a^i_{jk}$ and $u_i$ appearing in the presentation of $C$: write $r_i$ and $s_i$ for the image of $u_i$ and $\alpha_i$ under $y$. Since $y_K\in C_\varepsilon(K)$ we have that $(r_0\cdots r_{n-1})^h=\varepsilon p^m$ in $K$. There exists an index $j$ such that $r_j$ maps to zero in $k$ (otherwise $p$ would be invertible in $R$): define $x$ by setting $x(u_j)=\zeta r_j$, where $\zeta\in\mu_{p-1}$ is such that $\zeta^h=\varepsilon^{-1}$, $x(r_i)=y(r_i)$ for $i\neq j$ and $x(a^i_{jk})=y(a^i_{jk})$. Clearly $x_k=y_k$ since $r_j=0$ in $k$, and the fact that $x$ maps $(u_0\cdots u_{n-1})^h$ to $p^m$ means that $x_K\in C_1(K)$. This finishes the proof of the claim and of the lemma.
    \end{proofc}
    
\end{proof}

We are now going to prove that $\text{Spec}(C^{(r)})$ is not normal, except if $r=n-1$, in which case $C^{(n-1)}$ is a regular ring. Before doing so, it will be useful to have a combinatorial criterion to check whether the variables $\alpha_i$ and $\beta_i$ are zero on a point of $\text{M}^\text{GL}_{0,\Fp}$, compare with \cite[Lemma 3.1]{Sha}

\begin{proposition}
    Let $y:\emph{Spec}(k) \rightarrow \emph{M}^\emph{GL}_{0,\Fp}$ be a geometric point in the special fiber of the Iwahori local model and let $x$ be the alcove corresponding to the stratum it belongs to. Then
    \begin{align*}
        &(i) \:\:\alpha_i \neq 0 \:\:\Leftrightarrow \bar{\varphi}_{i} \:\:\text{is an isomorphism}  \:\:\Leftrightarrow \:\: 1-t^x_{i+1}(i+1)=0, \\
        &(ii) \:\: \beta_i \neq 0 \:\:\Leftrightarrow \varphi_{i _{|F_i}} \:\:\text{is an isomorphism} \:\:\Leftrightarrow \:\: t^x_i(i+1)=0.
    \end{align*}    
\end{proposition}
\begin{proof}
    The condition of $\bar{\varphi}_{i}$ and $\varphi_{i _{|F_i}}$ being isomorphisms or not is constant on the Schubert cells, hence we can assume that $y=F^x_\bullet$ is the base point given by some alcove $x$.
    \begin{itemize}
        \item [$(i)$] $\bar{\varphi}_i$ is an isomorphism if and only if 
        $$\text{dim}\left( \frac{k^n/F^x_{i+1}}{\text{im}(\bar{\varphi}_i)} \right)=\text{dim} \left( \frac{k^n}{\varphi_i(k^n)+F^x_{i+1}} \right)=0 $$
        Now, $\varphi_i(k^n)+F^x_{i+1}=\bigoplus_{j\neq i+1}k\text{e}_j + \bigoplus_{t^x_{i+1}(j)=1}k\text{e}_j = k^n$ holds if and only if $t^x_{i+1}(i+1)=1$. This proves the second equivalence, for the first see \cite[Lemma 3.15]{Sha}.
        \item [$(ii)$] Again, we only prove the second equivalence. In this case, $\varphi_{i_{|F_i}}$ is an isomorphism if and only if
        $$ \text{dim}\left( \frac{F^x_{i+1}}{\varphi_i(F^x_i)} \right) = 0 $$
        Since $F^x_i$ and $F^x_{i+1}$ both have dimension $r$ and $\varphi_i$ has a one-dimensional kernel, the above equality holds if and only if $F_i \cap \text{ker}(\varphi_i) = 0$, which in turn is equivalent to $t^x_i(i+1)=0$.
    \end{itemize}
\end{proof}

\begin{remark}
    It is shown in \cite{Goer} that the special fiber $\text{M}^\text{GL}_{0,\Fp}$ has $\binom{n}{r}$ irreducible components, which are parameterized by alcoves $x$ such that the difference vectors are constant, i.e. $t_i^x=t_j^x$ for all $i, j$. More precisely, each irreducible component of the special fiber is the closure of the stratum corresponding to an alcove with this property. We call such alcoves extreme and write $t^x$ for their difference vector. In this case, the fact whether $\alpha_i$ and $\beta_i$ are zero or not is determined by the $(i+1)$-th coordinate of $t^x$:
    \begin{align*}
        &t^x(i+1)=0 \,\Leftrightarrow \, \beta_i\neq 0,\, \alpha_i=0 \,\Leftrightarrow \, \bar{\varphi}_{i} \,\emph{is not an isomorphism} , \,\varphi_{i _{|F_i}}\, \emph{is an isomorphism},\\
        &t^x(i+1)=1 \,\Leftrightarrow \, \beta_i=0,\, \alpha_i\neq 0 \, \Leftrightarrow \,
        \bar{\varphi}_{i} \,\emph{is an isomorphism} , \,\varphi_{i _{|F_i}}\, \emph{is not an isomorphism}.
    \end{align*}
    In particular, either the restriction of $\varphi_i$ to $F_i$ or the map induced on the quotient has to be an isomorphism.
\end{remark}

\begin{theorem}\label{non normal}
    Suppose that $r\neq n-1$: then the local model $U_1$ is not normal. If $r=n-1$, then $U_1$ is regular.
\end{theorem}
\begin{proof}
    \begin{itemize}
        \item [$(i)$] Suppose that $r\neq n-1$. When $g=\text{gcd}(n-r,p-1)>1$, we have showed in Proposition \ref{geom structure U1} that all the irreducible components of $U_1$ intersect in the preimage $\uptau_1$ of the worst singular point, therefore $U_1$ is not normal. Now, if $g=1$ the local model $U_1$ is reduced and irreducible. Consider the following equalities between elements of $\text{Frac}(C)$
        \[
        \left(\frac{\beta_i u_i}{u_k}\right)^{p-1}=\beta_i^{p-1}\frac{\alpha_i}{\alpha_k}=\beta_i^{p-1}\frac{\beta_k}{\beta_i}=\beta_i^{p-2}\beta_k:
        \]
        if we show that $u_k$ does not divide $\beta_iu_i$ in $C$, the first part of the theorem is proved.
        \begin{claim}
            For all $k\neq i$, $u_k$ does not divide $\beta_iu_i$ in $C$.
        \end{claim}
        \begin{proofc}
            Seeking a contradiction, suppose that there exists $z\in C$ such that $\beta_i u_i=zu_k$: then $\beta_i u_i$ is zero in the quotient 
            \[ 
            D=\frac{C}{(u_0,\ldots, \hat{u}_i,\ldots, u_{n-1})}=\frac{(B/(\alpha_j)_{j\neq i})[u_i]}{(u_i^{p-1}-\alpha_i)}
            \]
            where $\hat{u}_i$ means that we are omitting $u_i$ in the generators of the ideal. But $D$ is a free $(B/(\alpha_j)_{j\neq i})$-module of rank $p-1$ and $u_i$ is part of a basis. If we show that $\beta_i \neq 0$ in $B/(\alpha_j)_{j\neq i}$ we will obtain a contradiction, hence proving the claim. We are going to determine an alcove $x$ on whose associated stratum $\alpha_j=\beta_j=0$ for all $j\neq i$, $\alpha_i=0$ and $\beta_i\neq 0$: in other words we are imposing that $\varphi_{i_{|F_i}}$ is an isomorphism and $\bar{\varphi}_i$ and all other $\varphi_{j_{|F_j}}$, $\bar{\varphi_j}$ are not. When $i=0$ this is equivalent, for the difference vectors and the vectors of the alcove itself, to have the following form:
            \begin{align*}
                t^x_0=&(0,*,*,\ldots,*,*,0), &x_0=(0,*,*,\ldots,*,*,0), \\
                t^x_1=&(0,1,*,\ldots,*,*,*), &x_1=(1,1,*,\ldots,*,*,*), \\
                &\vdots &\vdots \\
                t^x_{n-2}=&(*,*,*,\ldots,0,1,*), &x_{n-2}=(*,*,*,\ldots,1,1,*),\\
                t^x_{n-1}=&(*,*,*,\ldots,*,0,1), &x_{n-1}=(*,*,*,\ldots,*,1,1).
            \end{align*}
            Instead, when $i\neq0$ we get:
            \begin{align*}
                t^x_0=&(1,*,*,*,\ldots,*,*,*,0), &x_0=(1,*,*,*,\ldots,*,*,*,0), \\
                t^x_1=&(0,1,*,*,\ldots,*,*,*,*), &x_1=(1,1,*,*,\ldots,*,*,*,*), \\
                t^x_2=&(*,0,1,*,\ldots,*,*,*,*), &x_2=(*,1,1,*,\ldots,*,*,*,*),\\
                &\vdots  &\vdots \\
                t^x_i=&(*,*,\ldots,\overset{i}{0},\overset{i+1}{0},\ldots,*,*), &x_i=(*,*,\ldots,\overset{i}{1},\overset{i+1}{0},\ldots,*,*),\\
                &\vdots &\vdots \\
                t^x_{n-2}=&(*,*,*,*,\ldots,*,0,1,*), &x_{n-2}=(*,*,*,*,\ldots,*,1,1,*),\\
                t^x_{n-1}=&(*,*,*,*,\ldots,*,*,0,1), &x_{n-1}=(*,*,*,*,\ldots,*,*,1,1).
            \end{align*}           
            We write a choice for the alcove $x$ when $i=0$.
                \begin{align*}
                x_0=&(0,\overset{r}{\overbrace{1,\ldots,1,1}},0,\ldots,0), \\
                x_1=&(1,1,\ldots,1,1,0\ldots,0), \\
                x_2=&(1,1,\ldots,1,1,1\ldots,0), \\
                &\vdots\\
                x_{n-r}=&(1,1\ldots,1,1,1,\ldots, 1),\\
                x_{n-r+1}=&(1,2\ldots,1,1,1,\ldots, 1), \\
                &\vdots\\
                x_{n-1}=&(1,2\ldots,2,1,1,\ldots, 1).
            \end{align*}  
            This is enough to show that $C$ is not normal and the alcove for general $i$ is obtained similarly as for $i=0$. Note that when $r=n-1$ the difference vectors $t^x_i$ have only one zero component, hence this argument does not apply to this case (in fact, it will turn out that $u_k$ divides $\beta_i u_i$). This finishes the proof of the claim.
        \end{proofc}
        
        \item [$(ii)$] In the case $r=n-1$, $B=\Zp[a^i_k|i=0,\ldots, n-1; k=1,\ldots, n-1]/I$, where $I$ is generated by
        $$ a^{i+1}_{k-1}+a^{i+1}_{n-1}a^i_{k}\,, \:\: a^i_1a^{i+1}_{n-1}-p $$
        for $i=0,\ldots, n-1$ and $k=2,\ldots, n-1$. From these equations we see that 
        \[
        a^i_k=a^i_{n-1}a^{i-1}_{n-1}\dots a^{i-(n-1-k)}_{n-1} 
        \]
        so we can express all the variables in terms of $a^i_{n-1}$, and we are left with the only equation $a^0_{n-1}\dots a^{n-1}-p$. So $B$ is isomorphic to
        \[
        \Zp[x_0,\ldots,x_{n-1}]/(x_0\dots x_{n-1}-p) 
        \]
        and under this isomorphism we have $\alpha_i=a^{i+1}_{n-1}=x_i$ and $\beta_i=a^i_1=x_0\dots \hat{x}_i\dots, x_{n-1}$ ($i$-th term omitted). We now prove that $C=B[u_0,\ldots, u_{n-1}]/(u_i^{p-1}-x_i)_i$ is regular. Note that $C\cong \Zp[u_0\ldots,u_{n-1}]/((u_0\dots u_{n-1})^{p-1}-p)$: we proceed by induction on $n$.
        \begin{itemize}
            \item [$(i)$] When $n=1$, $\spec(\Zp[u]/(u^{p-1}-p))=\{ (0), (u,p)\}$, but $(u,p)=(u)$ in $C$, which hence is regular.
            \item [$(ii)$] Assuming that $C_n=\Zp[u_0\ldots,u_{n-1}]/((u_0\dots u_{n-1})^{p-1}-p)$ is regular, we want to show that $C_{n+1}=\Zp[u_0,\ldots,u_{n}]/((u_0\dots u_{n})^{p-1}-p)$ is regular as well. We can write
            $$\spec(C_{n+1})=D(u_0)\cup\dots \cup D(u_n)\cup V(u_0,\ldots,u_n,p),$$
            where $D(u_i)=\spec(\Zp[u_i^\pm,u_0,\ldots,u_{n}]/((u_0\dots u_{n})^{p-1}-p))\cong \Gm\,\sqcup \,\spec({C_n})$ (send $u_j$ to $u_ju_i^{-1}$ for any choice of $j\neq i$) is regular by induction hypothesis. Thus we only have to check that $(u_0,\ldots,u_n,p)$ can be generated by $n+1$ elements, but this is true since $(u_0,\ldots,u_n,p)=(u_0,\ldots,u_n)$.
        \end{itemize}
        \end{itemize}
\end{proof}

\begin{remark}
    Suppose that $g>1$: the argument of the claim in the first part of Theorem \ref{non normal} can be applied to show that, for any $i\neq k$, $u_k$ does not divide $\beta_i u_i$ inside $C_\varepsilon$. Together with the topological flatness of $C_\varepsilon$, one could hope to deduce that the same holds in the quotient $C_\varepsilon'=C_\varepsilon/(p\text{-torsion})$: however, this is not true in general, as one can check in the case when $n=3,r=1,p=3$. In this case in fact, the ring $C_1$ is
    \[
    \frac{\Zp[x_0,x_1,x_2,u_0,u_1,u_2]}{\left( \scalebox{1.3}{$\substack{x_0x_1x_2-p,\, u_0u_1u_2-p \\ u_0^{p-1}-x_1x_2,\, u_1^{p-1}-x_0x_2,\, u_2^{p-1}-x_0x_1}$}  \right)}
    \]
    and we claim that $u_ku_j-x_iu_i$, for $i,j,k$ all distinct, is a non-zero $p$-torsion element. But this means that $u_k$ divides $x_iu_i$ inside $C_1'$, so the argument of the claim fails in this case.
\end{remark}

\begin{remark} \label{map 1 to n-1}
    \begin{itemize}
        \item [$(i)$] When $r=1$ the equations of the local model are very similar to the case $r=n-1$, with the exception that $\alpha_i$ and $\beta_i$ are exchanged: $\alpha_i=x_0\dots \hat{x}_i\dots, x_{n-1}$ and $\beta_i=x_i$. Hence the local model is the spectrum of
        \[
        \frac{\Zp[x_0,\ldots, x_{n-1},v_0,\ldots,v_{n-1}]}{(x_0\cdots x_{n-1}-p, v_i^{p-1}-x_0\cdots \hat{x}_i\cdots x_{n-1})_i}.
        \]
        \item[$(ii)$] The fact that when $r=n-1$ the local model is regular was already known thanks to \cite[Corollary 3.4.3]{HR}. In fact, in this case $\frac{\beta_i u_i}{u_j}=\frac{x_0\cdots \hat{x}_i \hat{x}_j\cdots x_{n-1}u_j^{p-1}u_i}{u_j}$ already belongs to $C$, hence the argument above does not apply here.
        \item  [$(iii)$] The ring $C$ is Cohen-Macaulay for all $n$ and $r$, since it is finite and free as a $B$-module, which is Cohen-Macaulay thanks to \cite{He}. Therefore, it is not regular in codimension one for $r\neq n-1$.
    \end{itemize}
\end{remark}

\subsection{Normalization of local model in the ``bad'' \texorpdfstring{ Drinfeld case}{}}

In this section, we are going to compute the normalization of the local model in the case $r=1$, which amounts to finding the normalization of each of its irreducible components with their reduced scheme structure. Recall that they are indexed by the $g$-th roots of unity in $\Zp$, where $g=\text{gcd}(p-1, n-1)$, and that they are all (non canonically) isomorphic: given this, we are going to find the normalization of the irreducible component corresponding to $1\in \mu_g$. To simplify the notation, we are going to denote by $C_1'$ and $\tilde{C}_1$ the irreducible component and its normalization respectively: despite not knowing a presentation of $C_1'$, we can use the one of $C_1$ to describe the normalization. As explained in Remark \ref{map 1 to n-1} and Proposition \ref{non normal}, we can describe $C^{(1)}$ and $C^{(n-1)}$ (which were defined in Definition \ref{def C}) as follows:
\begin{align*}
    &C^{(1)}=\frac{\Zp[x_0,\ldots, x_{n-1},v_0,\ldots,v_{n-1}]}{(x_0\dots x_{n-1}-p, v_i^{p-1}-x_0\dots \hat{x}_i\dots x_{n-1})_i},\\
    \\
    &C^{(n-1)}=\frac{\Zp[x_0,\ldots, x_{n-1},u_0,\ldots,u_{n-1}]}{(x_0\dots x_{n-1}-p, u_i^{p-1}-x_i)_i}.
\end{align*}
Therefore, $C_1'$ is the quotient of 
\[
C_1=\frac{\Zp[x_i, v_i | i=0,\ldots, n-1]}{\left( \scalebox{1.2}{$\substack{ x_0\dots x_{n-1}-\,p\\
        v_i^{p-1}-\,x_0\dots \hat{x}_i\dots x_{n-1}\\
        (v_0\cdots v_{n-1})^{\frac{p-1}{g}}-\,p^{\frac{n-1}{g}} }$} \right)}
\]
by its $p$-torsion. We are first going to identify $\Tilde{C}_1$ with a subring of $C^{(n-1)}$ generated by suitable elements. This description is closely related to a Veronese subring of a polynomial ring over $\Zp$, for which we can give an explicit presentation in terms of generators and relations generalizing a result of \cite{Stu}, see Proposition \ref{presentation I} and the discussion preceding it. We will make use of this to find a presentation for $\Tilde{C}_1$. The idea is to consider the following map of $B$-algebras, where $B=\Zp[x_0,\ldots,x_{n-1}]/(x_0\cdots x_{n-1}-p)$:
\begin{align*}
        &C_1\xrightarrow{\makebox[0.9cm]{$f$}} C^{(n-1)}\\
        &v_i\mapsto u_0\dots\hat{u}_i\dots u_{n-1},
\end{align*}
which is finite and whose image turns out to be the subring mentioned above. Note that $f$ factors to a map $f':C_1'\longrightarrow C^{(n-1)}$ since $C^{(n-1)}$ does not have $p$-torsion.

\begin{lemma}\label{gen fibers}
    The maps $f$ and $f'$ are finite. Passing to the fraction fields, $f'$ induces the map
\begin{align*}
       \frac{\Qp(v_0,\ldots, v_{n-1})}{(v_0^{\frac{p-1}{g}}-p^{\frac{n-1}{g}})} &\xrightarrow{\makebox[1cm]{}} 
       \frac{\Qp(u_0,\ldots, u_{n-1})}{(u_0^{p-1}-p)}\\
       v_0 &\longmapsto u_0^{n-1} \\
       v_i &\longmapsto u_0/u_i,\quad i>0. 
    \end{align*}
In particular, $f'$ is injective and identifies $C^{(n-1)}$ with the normalization of $C_1'$ inside $\emph{Frac}(C^{(n-1)})$.
\end{lemma}
\begin{proof}
    Both $C_1$ and $C^{(n-1)}$ are finite over the Iwahori local model $B$, so $f$ and hence $f'$ are finite. We now compute the generic fibers of $C_1'$ and $C^{(n-1)}$:
    \[
     C'_{1,\Qp}=C_{1,{\Qp}} \,\overset{(1)}{\cong}\, \frac{\Qp[x_1^\pm,\ldots, x_{n-1}^\pm,v_0,\ldots,v_{n-1}]}{\left(
        \substack{v_0^{p-1}-x_1\cdots x_{n-1}\\
         v_i^{p-1}-p/x_i\\
         (v_0\cdots v_{n-1})^{\frac{p-1}{g}}-\,p^{\frac{n-1}{g}}
                }
         \right)}
     \,\overset{(2)}{\cong} \,
    \frac{\Qp[v_0^\pm,\ldots, v_{n-1}^\pm]}{(v_0^{\frac{p-1}{g}}-p^{\frac{n-1}{g}})},
    \]
    where the two isomorphisms are given by
    \[
    x_0 \,\overset{(1)}{\longmapsto}\, p/x_1\cdots x_{n-1}, \;\;\;\;\; x_i\,\overset{(2)}{\longmapsto}\, p/v_i^{p-1} \;\;\;\;\; v_0\,\overset{(2)}{\longmapsto}\, v_0/v_1\cdots v_{n-1}.
    \]
    For $C^{(n-1)}$ we get:
    \[
     C^{(n-1)}_{\Qp}\,\overset{(1)}{\cong}\, \frac{\Qp[u_0^\pm,\ldots,u_{n-1}^\pm]}{((u_0\cdots u_{n-1})^{p-1}-p)} \,\overset{(2)}{\cong}\, \frac{\Qp[u_0^\pm,\ldots, u_{n-1}^\pm]}{(u_0^{p-1}-p)},
    \]
    where the two isomorphisms are given by
    \[
    x_i\,\overset{(1)}{\longmapsto}\, u_i^{p-1}, \;\;\;\; u_0\,\overset{(2)}{\longmapsto}\, u_0/u_1\cdots u_{n-1}
    \]
    Under these identifications, the map $f'$ becomes
    \begin{align*}
       \frac{\Qp[v_0^\pm,\ldots, v_{n-1}^\pm]}{(v_0^{\frac{p-1}{g}}-p^{\frac{n-1}{g}})} &\xrightarrow{\makebox[1cm]{$f'_{\Qp}$}} 
       \frac{\Qp[u_0^\pm,\ldots, u_{n-1}^\pm]}{(u_0^{p-1}-p)}\\
       v_0 &\longmapsto u_0^{n-1} \\
       v_i &\longmapsto u_0/u_i,\quad i>0.
    \end{align*}
    Since $f'$ is a map between integral domains which extends to their fraction fields, it is injective. Since $C^{(n-1)}$ is normal we get the final part of the lemma.
\end{proof}

\begin{remark} \label{monomials}
\begin{itemize}
    \item [$(i)$]From the calculations made in the proof of Lemma \ref{gen fibers} we see that $C_{1,\Qp}$ is normal, hence it contains the normalization $\Tilde{C}_1$ of $C'_1$. We have the following commutative diagram:
    \[
    \begin{tikzcd}
    C_1 \arrow[rrd, "f", bend left] \arrow[rdd, bend right] \arrow[rd, two heads] &  &\\
    & C_1' \arrow[r, "f'", hook] \arrow[d, hook] & C^{(n-1)} \arrow[d, hook] \\
    & {C_{1,\Qp}} \arrow[r, hook, "f_{\Qp}"] & C^{(n-1)}_{\Qp},      
    \end{tikzcd}
    \]
    where we write $f_{\Qp}$ instead of $f'_{\Qp}$ since they are the same. Moreover, since $C^{(n-1)}$ is normal we have that $\Tilde{C}_1=f_{\Qp}^{-1}(C^{(n-1)})$. Equivalently, $\Tilde{C}_1$ can be described as a subring of $C^{(n-1)}$, namely as the intersection $f_{\Qp}(C_{1,\Qp})\cap C^{(n-1)}$: this is the approach we are going to take. We are also going to use the presentation of $C^{(n-1)}$ as 
    \[
    C^{(n-1)}=\frac{\Zp[u_0,\ldots, u_{n-1}]}{((u_0\cdots u_{n-1})^{p-1}-p)}.
    \]
    \item [$(ii)$] For $i=0,\ldots, n-1$, let $a_i, b_i\in \Z$ and consider the monomial in $C_{1,\Qp}$ given by $\underline{x}^{\underline{a}}\underline{v}^{\underline{b}}=x_0^{a_0}\cdots x_{n-1}^{a_{n-1}}\cdot v_0^{b_0}\cdots v_{n-1}^{b_{n-1}}$. These monomials form a set of generators (as a $\Qp$-vector space) of $C_{1,\Qp}$ and $f_{\Qp}$ evaluated on them gives: 
    \begin{equation}
        f_{\Qp}(\underline{x}^{\underline{a}} \underline{v}^{\underline{b}})=u_0^{a_0(p-1)+\beta-b_0}\cdots u_{n-1}^{a_{n-1}(p-1)+\beta-b_{n-1}}.
    \end{equation}
    Here $\beta = \sum_i b_i$, and setting $\alpha=\sum_ia_i$ we see that the degree of $f(\underline{x}^{\underline{a}} \underline{v}^{\underline{b}})$ is $\alpha(p-1)+\beta(n-1)$, a multiple of $g$.
\end{itemize}
\end{remark}

\begin{proposition}\label{normalisation1}
    The normalization of $C'_1$ is the subring of $C^{(n-1)}$ generated over $\Zp$ by the monomials $\underline{u}^{\underline{c}}=u_0^{c_0}\cdots u_{n-1}^{c_{n-1}}$, where $\underline{c}\in \mathbb{Z}_{\geq 0}^n$ is such that $\sum_i c_i=g$.
\end{proposition}
\begin{proof}
    We want to show that $f_{\Qp}(C_{1,\Qp})\cap C^{(n-1)}=\Zp[\underline{u}^{\underline{c}} \,|\,\underline{c}\in \mathbb{Z}_{\geq 0}^n,\, \sum_ic_i=g]$ as a subring of $C^{(n-1)}$: we prove the two inclusions.
    \begin{itemize}
        \item [$(\subseteq)$] By Remark \ref{monomials}, $f_{\Qp}(C_{1,\Qp})$ is generated as a $\Qp$-vector space by monomials of degree divisible by $g$, therefore $f_{\Qp}(C_{1,\Qp})\subseteq\Qp[\underline{u}^{\underline{c}} \,|\, \underline{c}\in \Z^n ,\,\sum_ic_i=g]$. Now, we have that $\Qp[\underline{u}^{\underline{c}} \,|\,\underline{c}\in \Z^n ,\,\sum_ic_i=g] \cap C^{(n-1)}= \Zp[\underline{u}^{\underline{c}} \,|\, \underline{c}\in \Z_{\geq 0}^n ,\, \sum_ic_i=g]$ and this proves the first inclusion.

        \item [$(\supseteq)$] We have to show that any $\underline{u}^{\underline{c}}$ with $\sum_ic_i=g$ and $c_i\geq 0$ is the image, under $f_{\Qp}$, of some element of $C_{1,\Qp}$. Write $g=\tau(p-1)+\sigma(n-1)$: we need to find $a_i, b_i \in \Z$ such that $c_i=a_i(p-1)+ \beta-b_i$, with $\alpha=\sum_ia_i$ and $\beta=\sum_ib_i$. Choose $a_i$ in such a way that $\alpha=\tau$, and set $b_i=a_i(p-1)+\sigma -c_i$. Then $\beta=\tau(p-1)+n\sigma-g=\sigma$ and this proves the second inclusion.
    \end{itemize}
\end{proof} 

\begin{notation}
    For the sake of clarity, we will use the following notations:
    \begin{itemize}
        \item $A=\{\underline{c}\in \Z^n_{\geq 0} \,|\, \sum_ic_i=g \}$;
        \item $\underline{u}=(u_0\cdots u_{n-1})$, $\underline{u}^{\underline{d}}=u_0^{d_0}\cdots u_{n-1}^{d_{n-1}}$ for $\underline{d}\in \Z^n$;
        \item $V=\Zp[\underline{u}^{\underline{c}} \,|\, \underline{c}\in A]$, as a subring of $\Zp[u_0\ldots, u_{n-1}]$.
    \end{itemize}
\end{notation}

\begin{remark}
    \begin{itemize}
        \item [$(i)$] Proposition \ref{normalisation1} shows that if $p-1$ and $n-1$ are coprime then $C^{(n-1)}$ is the normalisation of $C^{(1)}$, which is a domain in this case.        
        \item[$(ii)$]Note that Proposition \ref{normalisation1} implies that $\Tilde{C}_1$ coincides with the image of the subring $V \subset \Zp[v_0\ldots, v_{n-1}]$ under the projection to the quotient:
        \[
        \begin{tikzcd}
            V=\Zp[\underline{u}^{\underline{c}} \,|\, \sum_ic_i=g] \arrow[r, hook] \arrow[d, ->>] & \Zp[u_0\ldots, u_{n-1}] \arrow[d, ->>, "\pi"]\\
            \Tilde{C}_1 \arrow[r, hook] & \scalebox{1.2}{$\frac{\Zp[u_0\ldots, u_{n-1}]}{(\underline{u}^{p-1}-p)}$}
        \end{tikzcd}
        \]
    \end{itemize}
\end{remark}

We are now going to describe a presentation for $V$: this is done in \cite[Theorem 14.2]{Stu} for polynomials with coefficients in any field. We first briefly explain this result, then show how it holds for coefficients in any ring and apply it for $\Zp$. Let $k$ be any field and consider the Veronese subring $k[\underline{u}^{\underline{c}} \,|\, \underline{c}\in A]\subset k[u_0,\ldots,u_{n-1}]$. Consider the polynomial ring $k[u_{\underline{c}} \,|\, \underline{c}\in A]$, where we are adding one formal variable $u_{\underline{c}}$ for each element of $A$. There is a surjective map
\begin{align*}
    \varphi: k[u_{\underline{c}} \,|\, \underline{c}\in &A]\twoheadrightarrow k[\underline{u}^{\underline{c}} \,|\, \underline{c}\in A] \subset k[u_0\ldots, u_{n-1}]\\
    &u_{\underline{c}}\longmapsto \underline{u}^{\underline{c}}
\end{align*}
whose kernel is a toric ideal $I_{A}$, in the notation of \cite[Section 14.A]{Stu} (note that our set $A$ is a special case of the set $\mathcal{A}$ considered there, given by $r=g$, $d=n$ and $s_1=\cdots s_n=g$) and Theorem 14.2 of loc. cit. gives a reduced Gr{\"o}bner basis for $I_A$. However, Sturmfels uses a different notation for these monomials, which we now translate to our setting. First, there is a bijection, that we call str, between $A$ and the set of weakly increasing strings of length $g$ on the alphabet $\{1,\ldots, n\}$, under which the vector $\underline{c}=(c_0,\ldots, c_{n-1})\in A$ is mapped to the string
\begin{align*}
    \text{str}(\underline{c})=(\:&11\cdots 1\,&&22\cdots2\; &\cdots\; &&nn\cdots n)\\
    &c_0 \text{\footnotesize -times} &&c_1 \text{\footnotesize-times}&  &&c_{n-1}\text{\footnotesize-times}
\end{align*}
If $\text{str}(\underline{c})=z_1\cdots z_g$, write $u_{z_1\cdots z_g}$ for the monomial $u_{\underline{c}}$. Denote by sort the function which takes any string on the alphabet $\{1,\cdots, n\}$ and orders it in weakly increasing order. With these notations, the set of generators for $I_A$ given in the above cited theorem is:
\[
\{u_{a_1\cdots a_g}u_{b_1\cdots b_g}-u_{d_1\cdots d_g}u_{e_1\cdots e_g} \,|\, d_1e_1\cdots d_ge_g=\text{sort}(a_1b_1\cdots a_gb_g)\} 
\]

\begin{notation}\label{notation}
    We are going to use the following notation.
    \begin{itemize}
        \item Given two strings $a=a_1\cdots a_g$ and $b=b_1\cdots b_g$, write $a+b$ for the string $a_1b_1\cdots a_gb_g$.

        \item Given a string $w=w_1\cdots w_{2g}$, set odd$(w)=w_1w_{3}\cdots w_{2g-1}$ and even$(w)=w_2w_4\cdots w_{2g}$.

        \item Given $\underline{c}_1, \underline{c}_2\in A$, by writing $(\underline{c}_3, \underline{c}_4)=\text{sort}(\underline{c}_1, \underline{c}_2)$ we mean that
        \begin{itemize}
            \item[] $\underline{c}_3=\text{str}^{-1}\circ\text{odd}\circ \text{sort}(\text{str}(\underline{c}_1)+\text{str}(\underline{c}_2))$,

            \item[] $\underline{c}_4=\text{str}^{-1}\circ\text{even}\circ \text{sort}(\text{str}(\underline{c}_1)+\text{str}(\underline{c}_2))$. 
        \end{itemize}
    \end{itemize}
\end{notation}

\begin{example}
    Let us make an example to illustrate this notation. Suppose that $n=3$ and $r=1$, so that $g=\text{gcd}(n-r,p-1)=2$ and therefore $A=\{ \underline{c}\in\Z_{\geq 0}^3 \,|\, \sum_i c_i=2\}$. Pick the vectors $\underline{c}_1=(1,1,0)$ and $\underline{c}_2=(0,0,2)$, corresponding to the strings $\text{str}(\underline{c}_1)=(12)$ and $\text{str}(\underline{c}_2)=(33)$ on the alphabet $\{1,2,3\}$, so that $\text{str}(\underline{c}_1)+\text{str}(\underline{c}_2)=(1233)$, which is already sorted. Then we have that
    \begin{itemize}
        \item[$(a)$] $\underline{c}_3=\text{str}^{-1}\circ\text{odd}\circ \text{sort}(\text{str}(\underline{c}_1)+\text{str}(\underline{c}_2))=\text{str}^{-1}(13)=(1,0,1)$,

        \item[$(b)$] $\underline{c}_4=\text{str}^{-1}\circ\text{even}\circ \text{sort}(\text{str}(\underline{c}_1)+\text{str}(\underline{c}_2))=\text{str}^{-1}(23)=(0,1,1)$,
    \end{itemize}
    thus $((1,0,1),(0,1,1))=\text{sort}((1,1,0),(0,0,2))$.
\end{example}

We are going to need the next lemma later.

\begin{lemma}\label{sort}
    We have that $\emph{sort}(\underline{c}_1, \underline{c}_2)=\emph{sort}(\underline{c}_3, \underline{c}_4)$ if and only if $\underline{c}_1+\underline{c}_2=\underline{c}_3+\underline{c}_4$. In particular, if $(\underline{c}_3, \underline{c}_4)=\emph{sort}(\underline{c}_1, \underline{c}_2)$ then $\underline{c}_1+\underline{c}_2=\underline{c}_3+\underline{c}_4$.
\end{lemma}
\begin{proof}
    First, note that we have the following equivalence:
    \[\text{sort}(\underline{c}_1, \underline{c}_2)=\text{sort}(\underline{c}_3, \underline{c}_4) \quad\Longleftrightarrow\quad  \text{sort}(\text{str}(\underline{c}_1)+\text{str}(\underline{c}_2))=\text{sort}(\text{str}(\underline{c}_3)+\text{str}(\underline{c}_4)).
    \]
    Moreover, for $i=0,\ldots, n-1$, the number of occurrences of the symbol $i$ in the string $\text{sort}(\text{str}(\underline{c}_1)+\text{str}(\underline{c}_2))$, respectively in $\text{sort}(\text{str}(\underline{c}_3)+\text{str}(\underline{c}_4))$, is precisely $\underline{c}_1(i)+\underline{c}_2(i)$, respectively $\underline{c}_3(i)+\underline{c}_4(i)$. The first part of the lemma follows thanks to the fact that $\text{sort}(\text{str}(\underline{c}_1)+\text{str}(\underline{c}_2))=\text{sort}(\text{str}(\underline{c}_3)+\text{str}(\underline{c}_4))$ if and only if $\underline{c}_1(i)+\underline{c}_2(i)=\underline{c}_3(i)+\underline{c}_4(i)$ for all $ i=0,\ldots,n-1$. The second part follows from the fact that the sort function on $A\times A$ is idempotent, so we can apply it to the equality $(\underline{c}_3, \underline{c}_4)=\text{sort}(\underline{c}_1, \underline{c}_2)$ and then use the first part of the lemma.
\end{proof}

\begin{proposition}\label{presentation I}
 Let $R$ be any ring. The kernel $I_A$ of the map
\begin{align*}
    \varphi_R: R[u_{\underline{c}} \,|\, \underline{c}\in &A]\twoheadrightarrow R[\underline{u}^{\underline{c}} \,|\, \underline{c}\in A] \subset R[u_0\ldots, u_{n-1}]\\
    &u_{\underline{c}}\longmapsto \underline{u}^{\underline{c}}
\end{align*}
    admits the following set of generators (as an ideal): $\{u_{\underline{c}_1}u_{\underline{c}_2}-u_{\underline{c}_3}u_{\underline{c}_4} \,|\, (\underline{c}_3, \underline{c}_4)=\emph{sort}(\underline{c}_1, \underline{c}_2)\}.$
\end{proposition}
\begin{proof}
    We will first give a set of generators of ker$(\varphi_R)$ as an $R$-module, then use it to reduce the question to the case where $R$ is an integral domain, more precisely a polynomial ring over $\Z$, which we can prove by passing to the fraction field. Consider the map of sets
        \begin{align*}
            \rho:\Z^{A} &\longrightarrow \Z^n\\
            (a_{\underline{c}})_{\underline{c}\in A}&\longmapsto \sum_{\underline{c}\in A}a_{\underline{c}}\cdot \underline{c}
        \end{align*}
        We have that $\varphi_R(\prod_A u_{\underline{c}}^{a_{\underline{c}}})=\prod_A \underline{u}^{a_{\underline{c}}\cdot \underline{c}}=\underline{u}^{\rho(a_{\underline{c}})}$, where we write $(a_{\underline{c}})$ instead of $(a_{\underline{c}})_{\underline{c}\in A}$, see Notation \ref{notation}.\\
        \begin{claim}
            $\text{ker}(\varphi_R)=\langle \, \prod_A u_{\underline{c}}^{a_{\underline{c}}} - \prod_A u_{\underline{c}}^{b_{\underline{c}}} \;|\; \rho(a_{\underline{c}})=\rho(b_{\underline{c}}) \, \rangle$ as an $R$-module.
        \end{claim}
        \begin{proofc}
            Write $X_R=\langle \, \prod_A u_{\underline{c}}^{a_{\underline{c}}} - \prod_A u_{\underline{c}}^{b_{\underline{c}}} \;|\; \rho(a_{\underline{c}})=\rho(b_{\underline{c}}) \, \rangle_R$. The inclusion $\text{ker}(\varphi_R)\supseteq X_R$ is clear, so we prove the other one. Let $f\in \text{ker}(\varphi_R)$ and write in$(f)=r\cdot \prod u_{\underline{c}}^{a_{\underline{c}}}$ for the highest monomial of $f$ with respect to the lexicographic order on $A$ (i.e. $u_{\underline{c}} \leq u_{\underline{c}'}$ if $\underline{c}\leq \underline{c}'$ lexicographically). Suppose by contradiction that $f$ does not lie in $X_R$ and choose it with minimal initial term in$(f)$ (with respect to the lexicographic order). In the expansion of $\varphi_R(f)$ as a $R$-linear combination of monomials there is a term of the form $r\cdot u_{\underline{c}}^{\rho(a_{\underline{c}})}$ and, since $\varphi_R(f)=0$, $f$ has another monomial of the form $s\cdot \prod u_{\underline{c}}^{b_{\underline{c}}}$ such that $\rho(a_{\underline{c}})=\rho(b_{\underline{c}})$. Now consider $f'=f-r\cdot (\prod u_{\underline{c}}^{a_{\underline{c}}} - \prod u_{\underline{c}}^{b_{\underline{c}}})$: it lies in $\text{ker}(\varphi_R)$ but not in $X_R$ and in$(f')<\,$in$(f)$. This is a contradiction and the claim is proved. 
        \end{proofc}

        We now reduce to the case of a domain. Let $f\in \text{ker}(\varphi_R)$ and write it as an $R$-linear combination of binomials as in the claim: $f=\sum r_{(a_{\underline{c}},b_{\underline{c}})}(\prod_A u_{\underline{c}}^{a_{\underline{c}}} - \prod_A u_{\underline{c}}^{b_{\underline{c}}})$, where the set $D=\{((a_{\underline{c}}),(b_{\underline{c}}))\in \Z^A\times \Z^A \,|\, r_{(a_{\underline{c}},b_{\underline{c}})}\neq 0\}$ is finite. For clarity's sake, write $\Z[D]=\Z[x_{(a_{\underline{c}},b_{\underline{c}})} \,|\, ((a_{\underline{c}}),(b_{\underline{c}}))\in D]$ for the polynomial ring with one variable for each element of $D$. Consider the following commutative diagram:
        \[
        \begin{tikzcd}
            \text{ker}(\varphi_R)\subset R[u_{\underline{c}} \,|\, \underline{c}\in A] \arrow[r, "\varphi_R"] & R[\underline{u}^{\underline{c}} \,|\, \underline{c}\in A]\\
            \text{ker}(\varphi_{\Z[D]})\subset \Z[D][u_{\underline{c}} \,|\, \underline{c}\in A] \arrow[u, "\tau"] \arrow[r, "\varphi_{\Z[D]}"] &\Z[D][\underline{u}^{\underline{c}} \,|\, \underline{c}\in A] \arrow[u],
        \end{tikzcd}
        \]
        where the vertical maps are induced by the natural map $\Z[D]\rightarrow R$ sending $x_{(a_{\underline{c}},b_{\underline{c}})}$ to $r_{(a_{\underline{c}},b_{\underline{c}})}$. Let $f'=\sum x_{(a_{\underline{c}},b_{\underline{c}})}(\prod_A u_{\underline{c}}^{a_{\underline{c}}} - \prod_A u_{\underline{c}}^{b_{\underline{c}}}):$ it is an element of ker$(\varphi_{\Z[D]})$ which maps to $f$. Assuming that the proposition holds for integral domains, we can write $f'=\sum r'_{\underline{c}_i}(u_{\underline{c}_1}u_{\underline{c}_2}-u_{\underline{c}_3}u_{\underline{c}_4})$ for suitable $r'_{\underline{c}_i}\in \Z[D][u_{\underline{c}} \,|\, \underline{c}\in A]$ with $(\underline{c}_3, \underline{c}_4)=\text{sort}(\underline{c}_1, \underline{c}_2)$. Then $f=\tau(f')=\sum \tau(r'_{\underline{c}_i})(u_{\underline{c}_1}u_{\underline{c}_2}-u_{\underline{c}_3}u_{\underline{c}_4})$ is also generated by the binomials as in the statement of the proposition. This shows that $\text{ker}(\varphi)$ is contained in the ideal generated by $\{ u_{\underline{c}_1}u_{\underline{c}_2}-u_{\underline{c}_3}u_{\underline{c}_4} \,|\, (\underline{c}_3, \underline{c}_4)=\text{sort}(\underline{c}_1, \underline{c}_2)\}$, while the other inclusion follows from the fact that if $(\underline{c}_3, \underline{c}_4)=\text{sort}(\underline{c}_1, \underline{c}_2)$ then $\underline{c}_1+\underline{c}_2=\underline{c}_3+\underline{c}_4$, by Lemma \ref{sort}.\\        
        Now assume that $R$ is an integral domain and write $K$ for its fraction field: by \cite[Theorem 14.2]{Stu} we know that the proposition holds for $\varphi_K$. Clearly  any binomial $\prod_A u_{\underline{c}}^{a_{\underline{c}}} - \prod_A u_{\underline{c}}^{b_{\underline{c}}}$ of $X_R$ belongs to $\text{ker}(\varphi_{K})$ and we can apply \cite[Lemma 3.8]{Herz} to write it in the form
        \[
        \prod_A u_{\underline{c}}^{a_{\underline{c}}} - \prod_A u_{\underline{c}}^{b_{\underline{c}}} = \sum_{k=1}^s \pm \textbf{w}f_k,
        \]
        where \textbf{w} is a monomial in the variables $u_{\underline{c}}$ and $f_k$ are suitable elements of  $\{u_{\underline{c}_1}u_{\underline{c}_2}-u_{\underline{c}_3}u_{\underline{c}_4} \}_{sort}$: this implies that $\text{ker}(\varphi)$ is contained in the ideal generated by $\{u_{\underline{c}_1}u_{\underline{c}_2}-u_{\underline{c}_3}u_{\underline{c}_4} \}_{sort}$ and the other inclusion follows as in the previous paragraph.
\end{proof}

Now that we have a presentation for the ring $V$, we can use it to deduce one for $\Tilde{C}_1$ by determining a set of generators of the kernel of the map $\psi$ defined in the following diagram:
\[
\begin{tikzcd}
    \scalebox{1.2}{$\frac{\Zp[u_{\underline{c}} \,|\, \underline{c}\in A]}{(u_{\underline{c}_1}u_{\underline{c}_2}-u_{\underline{c}_3}u_{\underline{c}_4})_{sort}}$} \arrow[r, "\cong"] \arrow[rd, swap, "\psi"] &\Zp[\underline{u}^{\underline{c}} \,|\, \sum_ic_i=g] \arrow[d, ->>, "\pi"]\\
    &\Tilde{C}_1\;.
\end{tikzcd}
\]
We first determine the kernel of the projection $\pi$.

\begin{lemma}
    The kernel of $\pi$ is the ideal $(\underline{u}^{p-1}-p)V\subset V$.
\end{lemma}
\begin{proof}
    By definition, $\text{ker}(\pi)= V \cap (\underline{u}^{p-1}-p)$, where $(\underline{u}^{p-1}0p)$ is an ideal of $\Zp[u_i]$. As $\underline{u}^{p-1}-p \in V$, we see that $\text{ker}(\pi)\supseteq (\underline{u}^{p-1}-p)V$. For the other inclusion, let $(\underline{u}^{p-1}-p)f$ be an element of the kernel of $\pi$, for some $f=\sum_{\underline{b}\in B} a_{\underline{b}}\,\underline{u}^{\underline{b}}\;$, with $a_{\underline{b}}\neq 0$ for each $\underline{b}\in B$. Write 
    \begin{equation}\label{prod}
        (\underline{u}^{p-1}-p)f = \sum_{\underline{b}\in B} a_{\underline{b}}\,\underline{u}^{\underline{b}+(p-1)} - \sum_{\underline{b}\in B} p\cdot a_{\underline{b}}\,\underline{u}^{\underline{b}},
    \end{equation}    
     where $\underline{b}+(p-1)=(b_0+p-1,\ldots, b_{n-1}+p-1)$. We have to show that $f\in V$, i.e. that $g$ divides $\sum_i b_i$ for any $\underline{b}\in B$: fix such a $\underline{b}$. Let $m$ be the positive integer such that $\underline{b}+(m-1)(p-1)$ lies in $B$ but $\underline{b}+m(p-1)$ does not (it exists as $B$ is a finite set). Then $\underline{u}^{\underline{b}+m(p-1)}$ appears in the first summand of (\ref{prod}) but not in the second, so it appears in the expression of $(\underline{u}^{p-1}-p)f$ as a linear combination of monomials. Since $(\underline{u}^{p-1}-p)f$ belongs to $V$, $g$ divides the total degree of $\underline{u}^{\underline{b}+m(p-1)}$, which is $(\sum_ib_i)+mn(p-1)$, and therefore it also divides $\sum_i b_i$.
\end{proof}

To conclude, we need to find which element of $ \frac{\Zp[u_{\underline{c}} \,|\, \underline{c}\in A]}{(u_{\underline{c}_1}u_{\underline{c}_2}-u_{\underline{c}_3}u_{\underline{c}_4} )_{sort}}$ corresponds to $(u_0\cdots u_{n-1})^{p-1}-p$. Write $p-1=g\cdot h$ and consider the following elements of $A$:
\[
\underline{c}^{(i)}=(0,\ldots, g,\ldots, 0) \;\; \text{for} \; i=0,\ldots, n-1,
\]
where the only non zero entry is the $i$-th. Then $u_{\underline{c}^{(i)}}$ corresponds to $u_i^g$ inside $\Zp[u_0,\ldots, u_{n-1}]$ and hence $(u_{\underline{c}^{(0)}} \cdots u_{\underline{c}^{(n-1)}})^{h}-p$ corresponds to $(u_0\cdots u_{n-1})^{p-1}-p$. Therefore we get the following theorem.

\begin{theorem}\label{normalization C1}
    The normalization of one of the irreducible components of the pro-p Iwahori local model in the case of signature (1, n-1) is the spectrum of the ring
    \[
    \frac{\Zp[u_{\underline{c}} \,|\, \underline{c}\in A]}{\begin{pmatrix}
        u_{\underline{c}_1}u_{\underline{c}_2}-u_{\underline{c}_3}u_{\underline{c}_4} \,|\, (\underline{c}_3, \underline{c}_4)=\emph{sort}(\underline{c}_1, \underline{c}_2)\\
        (u_{\underline{c}^{(0)}} \cdots u_{\underline{c}^{(n-1)}})^{h}-p
    \end{pmatrix}}.
    \]
\end{theorem}

\begin{remark}
In the ideal defining the normalization of the local model we can replace the condition $(\underline{c}_3, \underline{c}_4)=\text{sort}(\underline{c}_1, \underline{c}_2)$ by $\text{sort}(\underline{c}_1, \underline{c}_2)=\text{sort}(\underline{c}_3, \underline{c}_4)$. A priori this would enlarge the set of generators, but the extra ones can be obtained by the previous ones. Indeed, suppose $\text{sort}(\underline{c}_1, \underline{c}_2)=\text{sort}(\underline{c}_3, \underline{c}_4)$ and let $(\underline{d}_1, \underline{d}_2)=\text{sort}(\underline{c}_1, \underline{c}_2)$: then we have $u_{\underline{c}_1}u_{\underline{c}_2}-u_{\underline{c}_3}u_{\underline{c}_4}= (u_{\underline{c}_1}u_{\underline{c}_2}-u_{\underline{d}_1}u_{\underline{d}_2})-(u_{\underline{c}_3}u_{\underline{c}_4}-u_{\underline{d}_1}u_{\underline{d}_2})$. By Lemma \ref{sort}, $\text{sort}(\underline{c}_1, \underline{c}_2)=\text{sort}(\underline{c}_3, \underline{c}_4)$ if and only if $\underline{c}_1+\underline{c}_2=\underline{c}_3+\underline{c}_4$. Therefore, another possible presentation of the normalization of $C_1$ is:
\[
\Tilde{C}_1\cong \frac{\Zp[u_{\underline{c}} \,|\, \underline{c}\in A]}{\begin{pmatrix}
        u_{\underline{c}_1}u_{\underline{c}_2}-u_{\underline{c}_3}u_{\underline{c}_4} \,|\, \scalebox{0.9}{$\underline{c}_1+\underline{c}_2=\underline{c}_3+\underline{c}_4$}\\
        (u_{\underline{c}^{(0)}} \cdots u_{\underline{c}^{(n-1)}})^{h}-p
    \end{pmatrix}}.
\]
\end{remark}

\subsection{The Haines-Stroh model}\label{HS non norm nor flat}

Recall the Iwahori local model $\text{M}_0^\text{GSp}$ in the Siegel case, as in the discussion after Definition \ref{GSp iwahori integral}. The condition (\ref{condition GSp loc model}) implies that $F_0$ and $F_n$ are isotropic with respect to the pairings $(\cdot, \cdot)$ and $p(\cdot, \cdot)$ respectively, and that $F_i$ determines $F_{2n-i}$ for $i=1,\ldots,n-1$. Hence, we can describe the points of $\text{M}_0^\text{GSp}$ by commutative diagrams of the form

\begin{equation}\label{local model siegel}
\begin{tikzcd}
    R^{2n} \arrow[r, "\varphi_0"] &R^{2n} \arrow[r, "\varphi_1"] &\cdots \arrow[r, "\varphi_{n-2}"] &R^{2n} \arrow[r, "\varphi_{n-1}"] &R^{2n}\\
    F_0 \arrow[r] \arrow[u, hook] &F_1 \arrow[r] \arrow[u, hook] &\cdots \arrow[r] &F_{n-1} \arrow[r] \arrow[u, hook] &F_n, \arrow[u, hook]
\end{tikzcd}
\end{equation}

where $\varphi_i=\text{diag}(1,\dots, p,\dots, 1)$ with $p$ in the $(i+1)$-th component and $F_i$ are direct summands of $R^{2n}$ of rank $n$ such that $F_0$ and $F_n$ are subject to the above conditions. Let $U\subset \text{M}_0^\text{GSp}$ be the open neighborhood of the worst singular point, namely the locus of those $F_\bullet \in \text{M}_0^\text{GSp}$ satisfying
\[
F_i\oplus \bigoplus_{j=n+i+1}^{2n+i}R \text{e}_j =R^{2n}
\]
for each $i=0,\ldots,n$. On this open we can fix a basis for each $F_i$ and write the chain as
    \[
    \begin{tikzcd}
        \setlength\arraycolsep{1pt}
        \scalebox{0.8}{$ \begin{pmatrix}
            1 &   & \\   
            &    & \\
            &    & 1  \\
            a^0_{11}  &\cdots & a^0_{1,n}\\
            \vdots & &\vdots\\
            a^0_{n,1}  &\cdots & a^0_{n,n}
        \end{pmatrix} $} \arrow[r, "\varphi_{0 _{|F_0}}"] 
        &
        \setlength\arraycolsep{1pt}
        \scalebox{0.8}{$ \begin{pmatrix}
            a^1_{n,1}  &\cdots & a^1_{n,n}\\
            1 &   & \\   
            &    & \\
            &    & 1  \\
            a^1_{11}  &\cdots & a^1_{1,n}\\
            \vdots & &\vdots            
        \end{pmatrix} $} \arrow[r, "\varphi_{1 _{|F_1}}"]
        &
        \cdots \arrow[r, "\varphi_{n-1 _{|F_{n-1}}}"]
        &
        \setlength\arraycolsep{1pt}
        \scalebox{0.8}{$ \begin{pmatrix}
            a^n_{11}  &\cdots & a^n_{1,n}\\
            \vdots & &\vdots\\
            a^n_{n,1}  &\cdots & a^n_{n,n} \\
            1 &   & \\   
            &    & \\
            &    & 1  
        \end{pmatrix} $}.
    \end{tikzcd}
    \]
The condition that $\varphi_i$ sends $F_i$ to $F_{i+1}$ can be expressed by $\varphi_i M_i=M_{i+1}A_i$ for some $n\times n$ matrix $A_i$, which is uniquely determined by $M_i, M_{i+1}$ and equal to 
\[
\begin{pmatrix}
0 & 1 &  & \\
  &   &  & \\
  &   & 0 & 1  \\
a^i_{11} &a^i_{12} &\cdots & a^i_{1n}
\end{pmatrix}.
\]
Then one gets the equations defining $U$ by comparing the entries of $\varphi_i M_i$ and $M_{i+1}A_i$ and by imposing the isotropicity conditions on $F_0$ and $F_n$, which take the form of:
\[
a^i_{j,k}=a^i_{n+1-k,n+1-j},
\]
for $i=0,n$ and $j,k=1,\ldots, n$. We write $U=\text{Spec}(B)$ and, in this section, as local model $\text{M}^\text{GSp}_1$ we use the root stack's atlas of type $(1)$.

\begin{definition}\label{def C Siegel}
    Haines-Stroh's local model in the Siegel pro-$p$ Iwahori case is given by $U_1=\text{Spec}(C)$, where
    \[
    C=\frac{B[u_0,v_0,\ldots,u_{n-1},v_{n-1}]}{\left( \scalebox{1.2}{$\substack{u_i^{p-1}-a^{i+1}_{nn}, \, v_i^{p-1}-a^i_{11}\\ u_0v_0-u_iv_i}$} \right)_i}.
    \]
\end{definition}

The next computation will be used to prove that the $\Gamma_1(p)$-local model (hence also the integral model) in the Siegel case is neither normal nor flat.

\begin{lemma}\label{identity for non flat non norm}
    The identity $a^1_{nn}a^0_{12}+a^2_{n-1,1}a^1_{11}=0$ holds in $B$.
\end{lemma}
\begin{proof}
    Comparing the (1,2) entry of the following matrix equation
    
    \begin{gather*}
        \scalebox{0.8}{$ \begin{pmatrix}
            p &   & \\   
            &    & \\
            &    & 1  \\
            a^0_{11}  &\cdots & a^0_{1,n}\\
            \vdots & &\vdots\\
            a^0_{n,1}  &\cdots & a^0_{n,n}
        \end{pmatrix} $}
        =
        \scalebox{0.8}{$ \begin{pmatrix}
            a^1_{n,1}  &\cdots & a^1_{n,n}\\
            1 &   & \\   
            &    & \\
            &    & 1  \\
            a^1_{11}  &\cdots & a^1_{1,n}\\
            \vdots & &\vdots            
        \end{pmatrix} $}
        \cdot
        \scalebox{0.8}{$ \begin{pmatrix}
            0 & 1 &  & \\
            &   &  & \\
            &   & 0 & 1  \\
            a^0_{11} &a^0_{12} &\cdots & a^0_{1n}
        \end{pmatrix}$}
    \end{gather*}
    we obtain that $0=a^1_{n1}+a^1_{nn}a^0_{12}$. Comparing the (1,1) entry of the following matrix equation 
    \begin{gather*}
        \scalebox{0.8}{$ \begin{pmatrix}
            a^1_{n1}  &\cdots & a^1_{nn}\\
            p &   & \\   
            &    & \\
            &    & 1  \\
            a^1_{11}  &\cdots & a^1_{1n}\\
            \vdots & &\vdots            
        \end{pmatrix} $}
        =
        \scalebox{0.8}{$ \begin{pmatrix}
            a^2_{n-1,1}  &\cdots & a^2_{n-1,n}\\
            a^2_{n1}  &\cdots & a^2_{nn}\\
            1 &   & \\   
            &    & \\
            &    & 1  \\
            a^2_{11}  &\cdots & a^2_{1n}\\
            \vdots & &\vdots            
        \end{pmatrix} $}
        \cdot
        \scalebox{0.8}{$ \begin{pmatrix}
            0 & 1 &  & \\
            &   &  & \\
            &   & 0 & 1  \\
            a^1_{11} &a^1_{12} &\cdots & a^1_{1n}
        \end{pmatrix}$}
    \end{gather*}
    we get that $a^1_{n1}=a^2_{n-1,1}a^1_{11}$. This proves the lemma.
\end{proof}

\begin{theorem}\label{Teo HS}
    If $n\neq 1$, the Haines-Stroh local model is not flat over $\Zp$, but it is topologically flat. For $n=1$, the local model is regular.
\end{theorem}
\begin{proof}
    If $n=1$, $B=\Zp[x,y]/(xy-p)$ and we have that
    \[
    C=\frac{B[u,v]}{(u^{p-1}-x,v^{p-1}-y)}=\frac{\Zp[u,v]}{((uv)^{p-1}-p)},
    \]
    which is regular. Suppose that $n>1$. For the non-flatness, we show that $C$ has $p$-torsion:
    \begin{align*}
        &p(a^0_{12}u_0u_1^{p-2}+a^2_{n-1, n}v_0^{p-2}v_1)=\,p(a^0_{12}u_0u_1^{p-2}) + \overset{p}{\overbrace{a^0_{11}a^1_{nn}}}a^2_{n-1,n}v_0^{p-2}v_1=\\ 
        &p(a^0_{12}u_0u_1^{p-2}) + a^0_{11}a^2_{n-1,n} \overset{a^1_{nn}}{\overbrace{u_0^{p-1}}} v_0^{p-2}v_1 = p(a^0_{12}u_0u_1^{p-2})+ a^0_{11}a^2_{n-1,n}u_0 \overset{(u_0v_0)^{p-2}}{\overbrace{(u_1v_1)^{p-2}}} v_1 =\\
        &u_0u_1^{p-2}(a^0_{11}a^1_{nn}a^0_{12}+a^0_{11}a^2_{n-1,n}\overset{v_1^{p-1}}{\overbrace{a^1_{11}}})=0,
    \end{align*}
where the last equality follows from Lemma \ref{identity for non flat non norm}. We need to show that $a^0_{12}u_0u_1^{p-2}+a^2_{n-1, n}v_0^{p-2}v_1\neq 0$ in $C$. If this was the case, we would obtain that $a^0_{12}u_0u_1^{p-2}=0$ in $C/(v_0,\ldots,v_n)$, which is isomorphic to
\[
\frac{B[u_0,\ldots, u_{n-1}]}{(a^i_{11}, \, u_i^{p-1}-a^{i+1}_{nn})_i}.
\]
This happens if and only if $a^0_{12}=0$ in $B/(a^0_{11},\ldots,a^{n-1}_{11})$, so the next claim proves the non-flatness assertion.

\begin{claim}
    $a^0_{12}\neq0$ in $B/(a^0_{11},\ldots,a^{n-1}_{11})$.
\end{claim}
\begin{proofc}
    It is easy to check that putting
    \[
    a^0_{12}=a^i_{n-i-1,n-i}=1 \quad \text{for} \, i=0,\ldots, n-2,
    \]
    and all the other variables equal to zero one gets an $\Fp$-valued point of $U$ on which $a^0_{11}=\cdots=a^{n-1}_{11}=0$ and $a^0_{12}\neq 0$. This finishes the proof of the claim.
\end{proofc}

For the topological flatness, we argue as in Lemma \ref{top flat unitary}; let us first write 
\[
    D=\frac{B[u_0,v_0,\ldots,u_{n-1},v_{n-1}]}{\left( \scalebox{1.2}{$\substack{u_i^{p-1}-a^{i+1}_{nn}, \, v_i^{p-1}-a^i_{11}}$} \right)_i}.
\]
Let $\bar{x}\in C(k)$ be a characteristic $p$ point of $C$: by the flatness of $D$, we find a mixed characteristic DVR $R$ with residue field $k$ and $x\in D(R)$ such that $x_k=\bar{x}$. Let $r_i,s_i\in R$ the images of $u_i,v_i$ under $x$. Inside the fraction field $K$ of $R$ we have, for each $i=1,\ldots,n-1$,
\[
\left( \frac{r_0s_0}{r_is_i} \right)^{p-1}=1
\]
thus there exists $\zeta_i\in\mu_{p-1}$ such that $r_0s_0=\zeta_i r_is_i$. Now, at least one between $r_i$ and $s_i$ has to be zero in $k$, say $r_i$: replacing it with $\zeta_ir_i$ gives us the desired $R$-point of $U_1$ lifting $\bar{x}$.

\end{proof}

\begin{remark}
    \begin{itemize}
        \item [$(i)$] Pappas had already checked with a computer that the Haines-Stroh local model for $\text{GSp}_4$ in the case when $p=5$ is not flat.
        \item [$(ii)$] From the topological flatness it follows that the $p$-torsion of $U_1$ is contained in its nilradical. On the other hand, the flat closure of $C$ embeds inside its generic fiber, which is reduced, hence the $p$-torsion coincides with the nilradical.
        \item [$(iii)$] In \cite{HS}, Haines and Stroh prove that their model is regular in codimension one. In light of the non reducedness, we can conclude that it is not $S_1$ (hence not Cohen-Macaulay). This is in contrast with Shadrach's model, which is finite free over the Iwahori model and therefore it is Cohen-Macaulay, but it is not regular in codimension one.

        \item [$(iv)$] It is an interesting question whether the flat closure of $C$ is normal. Since it coincides with its reduction, it is regular in codimension one and to check the $S_2$ property it is enough to show that the special fiber is $S_1$. However, computer evidence in the $\text{GSp}_4$ show that this is unlikely to be the case.

    \end{itemize}
\end{remark}

\end{document}